\documentclass[11pt,a4paper, twoside]{amsart}

\usepackage{amscd}
\usepackage[dvips]{graphicx}
\usepackage{wrapfig}
\theoremstyle{plain}
\numberwithin{equation}{section}
\newtheorem{theorem}{Theorem}[section]
\newtheorem{proposition}[theorem]{Proposition}
\newtheorem{corollary}[theorem]{Corollary}
\newtheorem{lemma}[theorem]{Lemma}
\newtheorem{conjecture}[theorem]{Conjecture}

\theoremstyle{definition}
\newtheorem{definition}[theorem]{Definition}
\newtheorem{remark}[theorem]{Remark}

\def\rank{\mathop{\mathrm{rank}}\nolimits}
\def\dim{\mathop{\mathrm{dim}}\nolimits}

\def\Ker{\mathop{\mathrm{Ker}}\nolimits}

\def\Hom{\mathop{\mathrm{Hom}}\nolimits}

\def\<{{\langle}}
\def\>{{\rangle}}

\def\Aut{\mathop{\mathrm{Aut}}\nolimits}
\def\Stab{\mathop{\mathrm{Stab}}\nolimits}
\def\Stabd{\mathop{\mathrm{Stab}^{\dagger}}\nolimits}
\def\tGL{\mathop{\widetilde{\mathrm{GL}}}\nolimits}
\def\+{\mathop{\oplus}\nolimits}

\def\Bigast{\mathop{\mbox{\huge $\ast$}}\nolimits}
\DeclareMathOperator*{\bigast}{\Bigast}

\newcommand{\mf}[1]{{\mathfrak{#1}}}
\newcommand{\bb}[1]{{\mathbb{#1}}}
\newcommand{\mca}[1]{{\mathcal{#1}}}
\newcommand{\mr}[1]{{\mathrm{#1}}}

\title[Hyperbolic metric and stability conditions]{A hyperbolic metric and stability conditions on K3 surfaces with $\rho=1$}
\author{Kotaro Kawatani}
\date{\today, version 5}
\address{Department of Mathematics, 
Nagoya University, 
Furoch, Chikusaku, Nagoya, Japan}
\email{kawatani@math.nagoya-u.ac.jp}

\subjclass[2010]{Primary 14F05; Secondly 14J28, 18E30, 32Q45}

\begin{document}
\maketitle

\section{Introduction}

In this article we introduce a hyperbolic metric on the (normalized) space of stability conditions on projective K3 surfaces $X$ with Picard rank $\rho (X) =1$. 
And we show that all walls are geodesic in the normalized space with respect to the hyperbolic metric. 
Furthermore we demonstrate how the hyperbolic metric is helpful for us by discussing mainly three topics. 
We first make a study of so called Bridgeland's conjecture. 
In the second topic we prove a famous Orlov's theorem without the global Torelli theorem. 
In the third topic we give an explicit example of stable complexes in large volume limits by using the hyperbolic metric. 
Though Bridgeland's conjecture may be well-known for algebraic geometers, we would like to start from the review of it.

\subsection{Bridgeland's conjecture}

In \cite{Bri07} Bridgeland introduced the notion of \textit{stability conditions} on arbitrary triangulated categories $\mca D$. 
By virtue of this we could define the notion of ``$\sigma$-stability'' for objects $E \in \mca D$ with respect to a stability condition $\sigma$ on $\mca D$.

Bridgeland also showed that each connected component of the space $\Stab (\mca D)$ consisting of stability conditions on $\mca D$ is a complex manifold 
unless $\Stab (\mca D)$ is empty.  
Hence the non-emptiness of $\Stab (\mca D)$ is one of the biggest problem. 
Many researchers study this problem in various situations. 
For instance suppose $\mca D$ is the bounded derived category $D(M)$ of coherent sheaves on a projective manifold $M$. 
In the case of $\dim M=1$, the non-emptiness of $\Stab(D(M))$ was proven in the original article \cite{Bri07}. 
Furthermore the space $\Stab(D(M))$ was studied in detail by \cite{Oka06} (the genus is $0$), \cite{Bri07} (the genus is $1$) and \cite{Mac07} (the genus is greater than $1$). 
In the case of $\dim M=2$, the non-emptiness was proven by \cite{Bri} (K3 or abelian surfaces) and \cite{ABL07} (other surfaces). 
In the case of $\dim M=3$ it is discussed by \cite{BMT11}. 
These are just a handful of many studies.

As we stated before, the space $\Stab (X)$ of stability conditions on the derived category $D(X)$ of a projective K3 surface $X$ is not empty by \cite{Bri}. 
This fact is proven by finding a distinguished  connected component $\Stabd (X)$. 
For $\Stabd (X)$
Bridgeland conjectured the following:

\begin{conjecture}[Bridgeland]\label{CONJ}
The space $\Stab (X)$ is connected, that is, $\Stab(X)= \Stabd(X)$. 
Furthermore the distinguished component $\Stabd(X)$ is simply connected. 
\end{conjecture}

As was proven by \cite{Bri} and \cite{HMS}, if the conjecture holds then we can determine the group structure of $\Aut (D(X))$ as follows:
We have the covering map $\pi \colon \Stabd (X) \to \mca P^+_0(X)$ by \cite[Theorem 1.1]{Bri} (See also Theorem \ref{covering}). 
Here $\mca P^+_0(X)$ is a subset of $H^*(X, \bb C)$ (See also Section \ref{termi}). 
By virtue of \cite{Bri} and \cite{HMS}, if Conjecture \ref{CONJ} holds we have the exact sequence of groups:
\begin{equation}
1 \to \pi _1 (\mca P^+_0(X)) \to \Aut (D(X)) \stackrel{\kappa}{\to} O^{+}_{\mr{Hodge}} (H^*(X, \bb Z)) \to  1, \label{exactseq} 
\end{equation}
where $O^{+}_{\mr{Hodge}} (H^*(X, \bb Z))$ is the Hodge isometry group of $H^*(X, \bb Z)$ preserving the orientation of $H^*(X,\bb Z)$. 
Hence Conjecture \ref{CONJ} predicts that the kernel $\Ker (\kappa)$ of the representation $\kappa$ is given by the fundamental group $\pi _1 (\mca P^+_0(X))$ and that $\Aut (D(X))$ is given by an extension of $\pi _1(\mca P^+_0(X))$ and $O^+_{\mr{Hodge}}(H^*(X, \bb Z))$.

\subsection{First theorem}

Recall the right $\tGL^+(2, \bb R) $-action on $\Stab (X)$ 
where $\tGL^+(2, \bb R)$ is the universal cover of $\mr{GL}^+(2, \bb R)$. 
We define $\Stab ^{\mr{n}}(X)$ by the quotient of $\Stabd (X)$ by the right $\tGL^+(2, \bb R)$ action. 
We call it a \textit{normalized stability manifold}. 
For a projective K3 surface with $\rho (X) =1$, 
we first introduce a hyperbolic metric on $\Stab ^{\mr{n}}(X)$. 
We also show that the hyperbolic metric is independent of the choice of Fourier-Mukai partners of $X :$

\begin{theorem}[=Theorem \ref{thm1}]
Assume that $\rho (X) =1$. 
\begin{itemize}
\item[(1)] $\Stab ^{\mr{n}} (X)$ is a hyperbolic 2 dimensional manifold. 
\item[(2)] Let $Y$ be a Fourier-Mukai partner of $X$ and $\Phi \colon D(Y) \to D(X)$ an equivalence which preserves the distinguished component $\Stabd(X)$. 
Then the induced morphism $\Phi _*^{\mathrm{n}} \colon \Stab ^{\mr{n}}(Y) \to \Stab ^{\mr{n}} (X)$ is an isometry with respect to the hyperbolic metric. 
\end{itemize}
\end{theorem}

Clearly if $\Stab(X)$ is connected it is unnecessary to assume that $\Phi$ preserves the distinguished component. 

We remark that there is another study by Woolf which focuses on the metric on $\Stab (\mca D)$ (not normalized!). 
In \cite{Woo}, he showed that $\Stab (\mca D)$ is complete with respect to the original metric introduced by Bridgeland. 
Our study is the first work which focuses on a different structure from Bridgeland's original framework.

\subsection{Second theorem}

Next, by using the hyperbolic structure, we observe the simply connectedness of $\Stabd (X)$ :

\begin{theorem}[=Theorem \ref{thm2}]
Let $X$ be a projective K3 surface with $\rho (X)=1$. 
The following three conditions are equivalent. 
\begin{itemize}
\item[(1)] $\Stabd(X)$ is simply connected. 
\item[(2)] $\Stab ^{\mr{n}}(X)$ is isomorphic to the upper half plane $\bb H$. 
\item[(3)] Let $W(X)$ be the subgroup of $\Aut (D(X))$ generated by two times compositions of the spherical twist $T_A$ by spherical locally free sheaves $A$. 
$W(X)$ is isomorphic to the free group generated by $T_A^2 $:
\[
W(X)= \bigast_{A} (\bb Z\cdot T_A^2)  ,  
\]
where $A$ runs through all spherical locally free sheaves and $\bigast$ is the free product. 
\end{itemize}
\end{theorem}

We give two remarks on Theorem \ref{thm2}. 
Firstly we could not prove the simply connectedness. 
However by using the hyperbolic structure on $\Stab ^{\mr{n}}(X)$, we can deduce the global geometry not only of $\Stab ^{\mr{n}}(X)$ but also of $\Stabd(X)$ as follows. 
Since $\Stabd (X)$ is a $\tGL^+(2, \bb R)$-bundle on $\Stab ^{\mr{n}}(X)$, 
and we see $\Stabd (X)$ is simply connected if and only if it is a $\tGL^+(2, \bb R)$-bundle over the upper half plane $\bb H$. 

Secondly, if Conjecture \ref{CONJ} holds then we see the kernel $\Ker (\kappa)$ is generated by $W(X)$ and the double shift $[2]$. 
Since the double shift $[2]$ commutes with any equivalence, the freeness of $W(X)$ implies $\Ker (\kappa) / \bb Z[2]$ is free. 
However in higher Picard rank cases, it is thought that the generators of $\Ker (\kappa)/ \bb Z [2]$ have relations (See  also Remark \ref{4.3}). 
Hence the freeness of $W(X)$ is a special phenomena.

\subsection{Third theorem}
 
In the third theorem, we study chamber structures on $\Stabd(X)$ in terms of the hyperbolic structure on $\Stab^{\mr{n}}(X)$. 
Before we state the third theorem, let us recall chamber structures. 

For a set $\mca S \subset D(X)$ of objects which has bounded mass and an arbitrary compact subset $B \subset \Stabd(X)$, we can define a finite collection of real codimension $1$ submanifolds $\{ W_{\gamma} \}_{\gamma \in \Gamma}$ satisfying the following property:
\begin{itemize}
\item Let $C \subset B \setminus \bigcup _{\gamma \in \Gamma}W_{\gamma}$ be an arbitrary connected component. If $E \in \mca S$ is $\sigma$-semistable for some $\sigma \in  C$ then $E$ is $\tau$-semistable for all $\tau \in C$. 
\end{itemize}
Each $W_{\gamma}$ is said to be a \textit{wall} and each connected component $C$ is said to be a \textit{chamber}. 
In this paper we call all data of chambers and walls a \textit{chamber structure}. 
We have to remark that chamber structures on $\Stabd(X)$ descend to the normalized stability manifold $\Stab^{\mr{n}}(X)$. 
Namely $C/\tGL^+(2, \bb R)$ and $\{ W_{\gamma}/\tGL^+(2, \bb R) \}$ also define a chamber structure on $\Stab^{\mr{n}}(X)$. 
Our third theorem is the following:

\begin{theorem}[=Theorem \ref{thm3}]
All walls of chamber structures of $\Stab^{\mr{n}}(X)$ are geodesic. 
\end{theorem}

\subsection{Revisit of Orlov's theorem}\label{motivation}
 
Generally speaking Fourier-Mukai transformations on $X$ may change chamber structures (This does not mean Fourier-Mukai transformations just permute chambers). 
By Theorems \ref{thm1} and \ref{thm3}, we see that the image of walls by Fourier-Mukai transformations is also geodesic in $\Stab^{\mr{n}}(X)$. 
Applying this observation we show the following:

\begin{proposition}[=Proposition \ref{Orlov}]
Let $X$ be a projective K3 surface with $\rho(X)=1$ and $Y$ a Fourier-Mukai partner of $X$ with an equivalence $\Phi \colon D(Y) \to D(X)$. 

If the induced morphism $\Phi_* \colon\Stab(Y) \to \Stab(X)$ preserves the distinguished component, then $Y$ is isomorphic to the fine moduli space of Gieseker stable torsion free sheaves.  
\end{proposition}

We have to mention that a more stronger statement was already proven by Orlov in \cite{OrlK3}; 
Any Fourier-Mukai partner of projective K3 surfaces is isomorphic to the fine moduli space of Gieseker stable sheaves. 
Our proof never needs the global Torelli theorem which was essential for Orlov's proof. 
Hence our proof gives a new feature of stability condition;
The theory of stability conditions substitutes for the global Torelli theorem. 
Since the strategy of Proposition \ref{Orlov} is technical, 
we will explain it in \S \ref{strategy}.

\subsection{Stable complexes in the large volume limit}

We also discuss the stability of complexes in large volume limits by using Lemma \ref{3.2} which is crucial for Theorem \ref{thm1}. 
More precisely in Corollary \ref{D3} we prove that the complexes $T_A(\mca O_x)$ are stable in the large volume limit where $T_A(\mca O_x)$ is a spherical twist of $\mca O_x$ by a spherical locally free sheaf. 
Originally it was expected that the $\sigma$-stability in the large volume limit is equivalent to Gieseker twisted stability (See also \cite[\S 14]{Bri}). 
However the possibility of stable complexes in the large volume limit is referred in \cite{Bay09}. 
We give an answer to this problem.

\subsection{Contents}

In Section \ref{2} we prepare some basic terminologies. 
In Section \ref{3} we prove the first main theorem. 
In Section \ref{4} we prove the second main theorem. 
The third theorem will be proven in Section \ref{5}. 
The analysis of $\partial U(X)$, which is necessary for Theorem \ref{thm2}, will be also done in Section \ref{5}. 
In Section \ref{6} we revisit Orlov's theorem. 
In Section \ref{7} we discuss the stability of $T_A^{-1}(\mca O_x)$ in the large volume limit.

\section{Preliminaries}\label{2}

In this section we prepare basic notations and lemmas. 
Let $(X,L)$ be a pari of a projective K3 surface with $\mr{NS}(X) = \bb Z L$. 
Almost all notions are defined for general projective K3 surfaces. 
To simplify the explanations we focus on K3 surfaces with $\rho (X) =1$. 

\subsection{Terminologies}\label{termi}

The abelian category of coherent sheaves on $X$ is denoted by $\mr{Coh}(X)$. 
Note that the numerical Grothendieck group $\mca N(X)$ is isomorphic to 
\[
H^0(X, \bb Z) \+ \mr{NS}(X) \+ H^4(X, \bb Z).
\] 
We put $v(E) = ch (E) \sqrt{td_X}$ for $E \in D(X)$. 
Then we see 
\[
v(E) = r_E \+ c_E \+ s_E \in \mca N(X).
\] 
One can easily check that $r_E= \rank E$, $c_E $ is the first Chern class $c_1(E)$ and $s_E = \chi (X, E)- \rank E$. 
Hence for a vector $v = r\+ c\+ s \in \mca N(X)$, the component $r$ is called the \textit{rank} of $v$.

The Mukai pairing $\<, \>$ on $H^*(X, \bb Z) $ is given by 
\[
\< r\+c \+ s , r'\+ c'\+ s' \> = c c' -r s' - r's. 
\]
By Riemann-Roch theorem we see 
\[
\chi(E,F) = \sum_{i} (-1)^i\dim \Hom_{D(X)}^i(E,F) = - \< v(E), v(F) \>. 
\]

An object $A \in D(X)$ is said to be \textit{spherical} if $A$ staisfies
\[
\Hom_{D(X)}^i(A,A) = \begin{cases} \bb C & (i=0,2) \\ 0 & (\text{otherwise}). \end{cases}
\]
We note that $v(A)^2=-2$ if $A$ is spherical. 
By the effort of \cite{ST}, for a spherical object $A$ we could define the autoequivalence $T_A$ called a \textit{spherical twist} (See also \cite[Chapter 8]{Huy}). 
By the definition of $T_A$ we have the following distinguished triangle for $E \in D(X)$:
\begin{equation}
\begin{CD}
\Hom_{D(X)}^*(A, E) \otimes A @>\rm{ev}>> E @>>> T_A(E) ,
\end{CD}\label{spherical triangle}
\end{equation}
where $\mr{ev}$ is the evaluation map. 
We call the above triangle a \textit{spherical triangle}. 
We note that the vector of $T_A(E)$ can be calculated as follows
\[
v(T_A(E)) = v(E)  + \< v(E), v(A)  \> v(A). 
\] 
Let $\Delta(X)$ be the set of $(-2)$-vectors:
\[
\Delta (X) =\{ \delta \in \mca N(X) | \delta ^2 =-2 \}
\]
and let $\Delta ^+(X) $ be the set $
 \{ \delta \in \Delta (X) |\delta = r \+ c \+ s, r>0  \}$.

Following \cite{Bri}, we put 
\[
\mca P(X) = \{ v \in \mca N(X)\otimes \bb C  | \mf{Re}(v) \mbox{ and }\mf{Im}(v) \mbox{ span a positive 2 plane} \}
\]
Since $\mca P(X)$ has two connected components, we define $\mca P^+(X)$ by the connected component containing $\exp (\sqrt{-1}\omega)$ where $\omega$ is an ample class. 
Then $\mca P^+(X)$ has the right $\mr{GL}^+(2, \bb R)$ action as the change of basis of the planes. 
This action is free. 
Hence there exists the quotient $\mca P^+(X) \to \mca P^+(X) /\mr{GL}^+(2, \bb R)$ which gives a principle $\mr{GL}^+(2, \bb R)$-bundle with a global section. 

Under the assumption $\rho(X)=1$,  
$\mca P^+(X)/\mr{GL}^+(2, \bb R)$ is isomorphic to the set $\mf{H}(X)$ where 
\[
\mf{H}(X) =  \{   (\beta,\omega)=(xL, yL) | x+\sqrt{-1} \in \bb H\}. 
\] 
Clearly $\mf{H}(X)$ is canonically isomorphic to $\bb H$. 
Then the global section $\mf{H}(X) \to \mca P^+(X)$ is given by 
\[
\mf{H}(X)  \ni (x,y) \mapsto \exp(\beta + \sqrt{-1}\omega) \in \mca P^+(X). 
\]
In particular $\mca P^+(X)$ is isomorphic to $\bb H \times GL^+(2, \bb R)$. 
We put $\mca P^+_0(X)$ by 
\[
\mca P^+_0(X) = \mca P^+(X) \setminus \bigcup _{\delta \in \Delta (X) }\< \delta \>^{\perp}
\]
where $\< \delta \>^{\perp}$ is the orthogonal complement of $\delta$ with respect to the Mukai pairing on $H^*(X, \bb Z)$\footnote{We remark that the definition of $\mca P^+_0(X)$ is independent of the assumption $\rho(X)=1$. }. 
Define
\[
\mf{H}_0(X) = \{ v \in \mf{H} (X) | \< \exp(v), \delta \> \neq 0\ (\forall \delta \in \Delta (X)) \}. 
\]
Then we see $\mca P^+_0(X)$ is isomorphic to $\mf{H}_0(X) \times \mr{GL}^+(2, \bb R)$.

\subsection{Stability conditions on K3 surfaces}\label{SK3}

Let $\Stab (X)$ be the set of numerical locally finite stability conditions on $D(X)$. 
We put $\sigma = (\mca A, Z) \in \Stab (X)$ where $\mca A$ is the heart of a bounded t-structure on $\mca D$ and $Z$ is a central charge. 
Since the Mukai paring is non-degenerate on $\mca N(X)$ we have the natural map:
\[
\pi \colon \Stab (X) \to \mca N(X)\otimes \bb C,\ \pi(\sigma)= Z^{\vee}
\]
where $Z(E) = \< Z^{\vee}, v(E)\>$.

In $\Stab (X)$, there is a connected component $\Stabd (X)$ which contains the set $U(X):$  
\begin{multline*}
U(X) =\{ \sigma  = (\mca A, Z) \in \Stab (X) |Z^{\vee} \in \mca P(X)\setminus \bigcup _{\delta \in \Delta (X) }\< \delta \>^{\perp}, \\ \mca O_x \mbox{ is $\sigma$-stable in the same phase for all }x \in X  \}. 
\end{multline*}
Let $\bar U(X)$ be the closure of $U(X)$ in $\Stab (X)$. 
Then we see that $\bar U(X)$ be the set of stability conditions $\sigma$ such that $\mca O_x$ ($\forall x\in X$) is $\sigma$-semistable in the same phase with $Z^{\vee} \in \mca P(X) \setminus \bigcup _{\delta \in \Delta (X) }\< \delta \>^{\perp}$. 
Define $\partial U(X)$ by $\bar U(X) \setminus U(X)$ and call it the \textit{boundary of $U(X)$}. 

We define the set $V(X)$ by
\[
V(X) = \{ \sigma= (\mca A, Z) \in U(X) | Z(\mca O_x)=-1, \ \mca O_x\mbox{ is $\sigma$-stable with phase $1$}  \}. 
\]

One can see $U(X) = V(X) \cdot \tGL^+(2, \bb R) \cong V(X) \times \tGL^+(2, \bb R)$ by \cite[Proposition 10.3]{Bri}. 
Furthermore the set $V(X)$ is parametrized by $(\beta, \omega) \in \mf{H}(X)$ in the following way:

For the pair $(\beta, \omega)$, put $\mca A_{(\beta, \omega)}$ and $Z_{(\beta, \omega)}$ as follows $:$
\begin{eqnarray*}
\mca A_{(\beta,\omega )} &:=&   \big\{ E^{\bullet} \in D(X) \big| 
H^i(E^{\bullet})	\begin{cases}
					\in \mca T_{(\beta,\omega)} & (i=0) \\ 
					\in \mca F_{(\beta,\omega)} & (i=-1) \\ 
					= 0 & (\text{otherwise}) 
					\end{cases} \big\} \\	
Z_{(\beta, \omega)} (E)	&:=& \< \exp(\beta+ \sqrt{-1}\omega), v(E) \>,
\end{eqnarray*}
where 
\begin{eqnarray*}
\mca T_{(\beta,\omega )} &:=& \{ E \in \mr{Coh}(X) | E\mbox{ is a torsion sheaf or }
\mu _{\omega}^-(E/\mr{torsion} ) > \beta \omega  \} \mbox{ and }\\
\mca F_{(\beta,\omega)} &:=& \{ E \in \mr{Coh}(X) | E\mbox{ is torsion free and } 
\mu _{\omega}^+ (E) \leq \beta \omega  \}. 
\end{eqnarray*}
Here $\mu_{\omega}^+ (E)$ (respectively $\mu_{\omega}^-(E)$) is the maximal slope 
(respectively minimal slope) of semistable factors of a torsion free sheaf $E$ with respect to the slope stability. 
Since the pair $(\mca T_{(\beta, \omega)}, \mca F_{(\beta, \omega)})$ gives a torsion pair on $\mr{Coh}(X)$, $\mca A_{(\beta, \omega)}$ is the heart of a bounded $t$-structure on $D(X)$.  
We denote the pair $(\mca A_{(\beta, \omega)}, Z_{(\beta, \omega)})$ by $\sigma _{(\beta, \omega)}$.

\begin{proposition}[{\cite[Proposition 10.3]{Bri}}]\label{2.1}
Assume that $(\beta, \omega)$ satisfies the condition
\begin{equation}
\< \exp(\beta+\sqrt{-1}\omega), \delta  \> \not\in \bb R_{\leq 0}, (\forall \delta \in \Delta^+(X)) \label{BG-condition}
\end{equation}
Then the pair $\sigma_{(\beta, \omega)}$ gives a numerical locally finite stability condition on $D(X)$.  
Furthermore we have
\[
V(X) = \{ \sigma_{(\beta, \omega)} \in \Stabd(X) | (\beta, \omega)\mbox{ satisfies the condition (\ref{BG-condition})} \}. 
\] 
\end{proposition}

\begin{remark}\label{Kaw4.0}
We put $v(E) = r_E \+ c_1(E) \+ s_E$ for $E \in D(X)$ . 
As the author remarked in \cite[Section 4, (4.1)]{Kaw10}, for objects $E \in D(X)$ with $\rank E \neq 0$, 
we can rewrite $Z_{(\beta , \omega)} (E)$ as follows, 
\begin{equation}
Z_{(\beta, \omega)}(E) = \frac{v(E)^2}{2 r_E} + \frac{r_E}{2}\Big( \omega + \sqrt{-1} \big(\frac{c_1(E)}{r_E}- \beta \big)   \Big)^2 \label{key}. 
\end{equation}
This equation (\ref{key}) plays an important role in Lemma \ref{3.2} which is crucial for Theorem \ref{thm1}. 
\end{remark}

\begin{definition}\label{W(X)}
For a projective K3 surface with $\rho (X)=1$ we define the subgroup $W(X)$ of $\Aut (D(X))$ generated by 
\[
W(X) = \<  T_A^2 | A=\mbox{spherical locally free sheaf}\>. 
\]
\end{definition}

Then by using $U(X)$ and $W(X)$ we can describe $\Stabd(X)$ in a explicit way:

\begin{proposition}[{\cite[Proposition 13.2]{Bri}}]\label{3.3}
Let $X$ be a projective K3 with $\rho (X)=1$. 
The distinguished connected component $\Stabd(X)$ is given by 
\[
\Stabd (X) = \bigcup _{\Phi \in W(X)} \Phi _* (\bar U(X)). 
\]
\end{proposition}

\begin{theorem}[\cite{Bri}]\label{covering}
The natural map $\pi: \Stabd(X) \to \mca N(X) \otimes \bb C$ has the image $\mca P^+_0(X)$. 
Furthermore $\pi$ is a Galois covering. 
The covering transformation group is the subgroup generated by equivalences in $\Ker (\kappa) $ which preserve $\Stabd(X)$.  
\end{theorem}

\begin{corollary}\label{covering2}
For a pair $(X,L)$, the induced map 
\[
\pi^{\mr{n}}\colon \Stab^{\mr{n}}(X) \to \mf{H}^+_0(X)
\]
is also a Galois covering map. 
\end{corollary}

\begin{proof}
We have the following $\mr{GL}^+(2, \bb R)$-equivariant diagram:
\[
\begin{CD}
\Stabd(X)/ \bb Z[2] @>\pi'>> \mca P^+_0(X) \\
@VVV @VVV\\
\Stab^{\mr{n}}(X) @>\pi^{\mr{n}}>> \mf{H}^+_0(X) . 
\end{CD}
\] 
We note that both vertical maps are $\mr{GL}^+(2, \bb R)$-bundles and that $\pi'$ is also a Galois covering.

By Theorem \ref{covering} the covering transformation group of $\pi '$ is a subgroup of $\Aut (D(X))/ \bb Z[2]$. 
Hence the right $\mr{GL}^+(2, \bb R)$-action on $\Stabd(X)/ \bb Z[2]$ commutes with the covering transformations. 
Hence $\pi^{\mr{n}}$ is also a Galois covering. 
\end{proof}

\subsection{On the fundamental group of $\mca P^+_0(X)$}\label{pi_1}

We are interested in the fundamental group $\pi _1 (\mca P^+_0(X))$. 
Generally speaking, it is highly difficult to describe the above condition (\ref{BG-condition}) explicitly. 
Because of this difficulty, it becomes difficult to determine the relation between generators of $\pi _1(\mca P^+_0(X))$. 
Hence it seems impossible to determine the group structure of $\pi _1(\mca P^+_0(X))$. 
However, under the assumption $\rho (X)=1$ it becomes easier. 



\begin{definition}\label{ass.pt}
Let $\delta = r\+ c \+ s \in \Delta(X)$. 
An associated point $p \in \mf{H}(X)$ with $\delta  \in \Delta(X)$ is the point $p \in \mf{H}(X)$ such that  
$\<\exp(p),\delta \> = 0$. 
We also denote the point by $p(\delta)$ and call it a \textit{spherical point}. 
If $\delta$ is the Mukai vector of a spherical object $A$ we denote simply $p(v(A))$ by $p(A)$.  
\end{definition}

\begin{remark}\label{ass.pt.rmk}
Let $\delta \in \Delta (X)$ and we put $\delta = r\+ c\+ s$. 
Since $c^2 \geq 0$ we see $r \neq 0$. 
Thus we have the disjoint sum $\Delta (X) = \Delta ^+(X) \sqcup (-\Delta^+ (X))$. 

Now we have the explicit description of $p(\delta ) $ as follows:
\[
p (\delta )= (\frac{c}{r},\frac{1}{\sqrt{d}|r|} L) \in \mf{H}(X), 
\]
where we put $L^2 =2d$. 
Moreover one sees $p(\delta)= p(-\delta)$. 
\end{remark}

The key lemma of this subsection is that the set $\{ p(\delta) \in \mf{H}(X) | \delta \in \Delta (X) \}$ is discreet in $\mf{H}(X)$. 
To show this claim we introduce some notations. 

\begin{definition}
Let $\delta = r\+ c \+ s \in \Delta ^+(X)$. 

1. 
We define the set $\Delta ^{(i)}(X)$ by
\[
\Delta ^{(i)}(X) = \{  r\+ c \+ s \in \Delta ^+(X) | r\mbox{ is the $i$-th smallest in }\Delta^+(X) \}. 
\]
We also define the rank associated to $\Delta ^{(i)}(X)$ by $r$ for some $\delta = r\+ c \+ s \in \Delta ^{(i)}(X)$. 

2. 
We define the subset $\mca V(X)$ of $\mf{H}(X)$ as follows. 
\[
\mca V(X) = \{ (\beta, \omega) \in \mf{H}(X) | (\beta, \omega)\mbox{ satisfies the condition }(\ref{BG-condition})  \}. 
\]
As we remarked in Proposition \ref{2.1} this set is isomorphic to $V(X)$ consisting of stability conditions by the natural morphism $\pi$. 

3. 
Let $r_i$ be the rank associated to $\Delta ^{(i)}(X)$. 
We define the subset $\mca V^{(i)}(X)$ of $\mca V(X)$ by
\[
\mca V^{(i)}(X) = \{ (\beta, \omega) \in \mca V(X) | \omega ^2 > \frac{2}{r_i^2} \}. 
\]
\end{definition}

\begin{remark}\label{spherical sheaf}
Let $X$ be a projective (not necessary Picard rank one) K3 surface. 
For any $\delta =r\+ c\+ s \in \Delta (X)$ with $r\geq 0$, there exists a spherical sheaf $A$ on $X$ such that $v(A)= \delta$ by \cite{Kul90}.  
In particular if $r>0$ then we can take $A$ as a locally free sheaf. 
In addition if we assume $\mr{NS}(X) = \bb ZL$ then we see $A$ is Gieseker-stable by \cite[Proposition 3.14]{Muk}. 
Since we see $\gcd(r, n)=1$ where $n$ satisfies $nL =c$, $A$ is $\mu$-stable by \cite[Lemma 1.2.14]{HL}. 
\end{remark}

\begin{remark}
For instance $\Delta ^{(1)}(X)$ is the set of Mukai vectors of line bundles on $X$. 
Thus $\rank \Delta ^{(1)}(X)=1$ for any $(X,L)$. 
However for $i>1$, the rank of $\Delta ^{(i)}(X)$ depends on the degree $L^2$. 

Since $\rank \Delta ^{(1)}(X)=1$, we see $(\beta,\omega)$ is in $\mca V^{(1)}(X)$ if and only if $\omega ^2 >2$. 
We have the following infinite filtration of $\mca V^{(i)}(X)$ ($i =1,2 ,3 \cdots $)
\[
\mca V^{(1)}(X) \subset \mca V^{(2)}(X) \subset \cdots \subset \mca V^{(n)}(X) \subset \cdots \subset \mca V(X). 
\]
\end{remark}

\begin{lemma}\label{key lemma}
Notations being as above, 
	\begin{itemize}
		\item[(1)] the set $\mf S = \{ p(\delta) \in \mf{H}(X) | \delta \in \Delta (X)  \}$ is a discreet set in $\mf{H}(X)$. 
		\item[(2)] Furthermore the set $\mca V(X)$ is open in $\mf{H}(X)$. 
	\end{itemize}
\end{lemma}

\begin{proof}
Suppose that $\mr{NS}(X) = \bb Z L $ with $L^2=2d$. 
Let $p(\delta) $ be the spherical point of $\delta \in \Delta ^+(X)$. 
We put $\delta = r\+ c \+ s$ where $c = n L$ for some $n \in \bb Z$. 

Recall that $p(\delta ) $ is given by 
\[
p(\delta) = (\frac{nL}{r}, \frac{1}{\sqrt{d}r}L). 
\]
We also note that $\gcd (r,n)=1$ since $\delta ^2 =-2$ and $\mr{NS}(X) = \bb Z L$. 
Let $B_{\epsilon}$ be the open ball  whose center is $p(\delta)$ and the radius is $\epsilon$ (with respect to the usual metric). 
Since $r_{i+1} \geq r_i +1$ (where $r_i$ is the rank of $\Delta ^{(i)}(X)$) if $\epsilon$ is smaller than $\frac{1}{\sqrt{d}} (\frac{1}{r}- \frac{1}{r +1})$ we see $B_{\epsilon} \cap \mf{S} = \{ p(\delta) \}$.

We prove the second assertion. 
We define $S(\delta)$ for $\delta \in \Delta ^+(X)$ as follows:
\[
S(\delta) = \{ (\beta, \omega) \in \mf{H}(X) | \beta = \frac{c}{r}, 0< \omega ^2 \leq \frac{2}{r^2} \}. 
\]
Then one can check that 
\[
\mca V(X) = \mf{H}(X) \setminus \bigcup _{\delta \in \Delta ^+(X)} S(\delta ). 
\]
Hence we see 
\[
\mca V^{(i)}(X) = \{ (\beta, \omega) \in \mf{H}(X)  | \omega ^2 > \frac{2}{r_i^2 }  \} \setminus \bigcup _{\delta \in \Delta ^{(\leq i-1)} }S(\delta),  
\]
where $\Delta ^{(\leq i )} = \bigcup _{j=1}^i \Delta ^{(j)}(X)$. 
Since the set 
\[
\{ \frac{c}{r} | \delta = r \+ c \+ s \in \Delta ^{(\leq i)} \}
\]
is discreet in $\bb R L$, the set $\mca V^{(i)}(X)$ is open in $\mf{H}(X)$. 
Since we have 
\[
\mca V(X) = \bigcup _{i \in \bb N} \mca V^{(i)}(X),
\]
the set $\mca V(X)$ is open in $\mf H (X)$. 
\end{proof}

\begin{definition}\label{generator}
We set elements of the fundamental groups $\pi _1(\mf{H}_0(X))$ and of $\pi _1(GL^+(2, \bb R))$ as follows. 
\begin{itemize}
\item We define $\ell _{\delta } $ by the loop which turns round only the spherical point $p(\delta ) \in \mf{H}(X)$ counterclockwise;
	\begin{figure}[htbp]
		\begin{center}
		\includegraphics[height=23mm, clip]{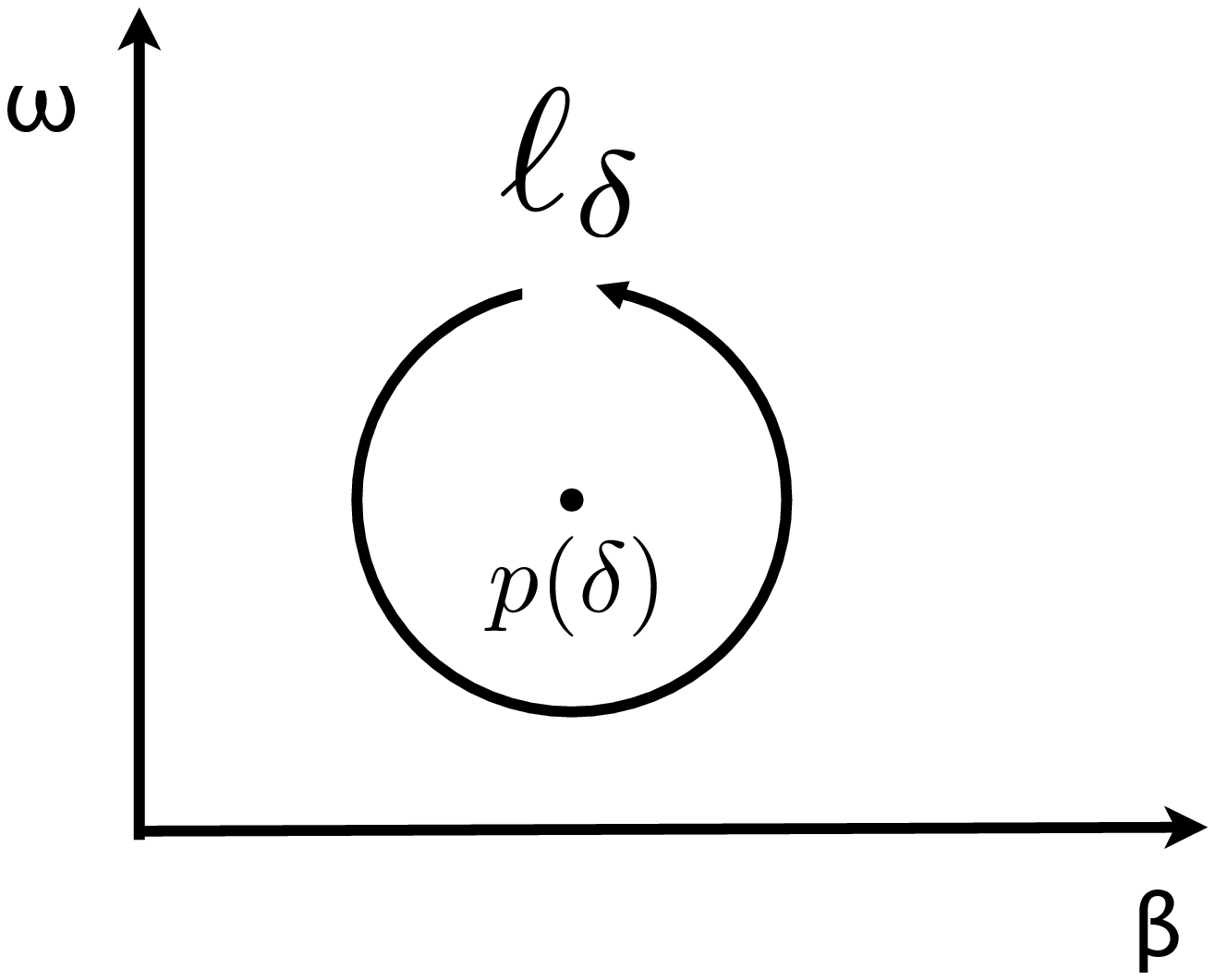}
		\end{center}
		\caption{For $p(\delta)$ 
we define the loop $\ell _{\delta}$ as the above direction. We also assume that there are no spherical points $p(\delta ')$ in the inside of $\ell _{\delta}$ except for $p(\delta)$ itself.  
}
		\label{fig:1}
	\end{figure}	
\item We define $g \in \pi _1(GL^+(2, \bb R))$ by 
\[
g\colon[0,1]\ni t \mapsto \begin{pmatrix} \cos (2 \pi t) & -\sin (2\pi t) \\ \sin (2 \pi t) & \cos (2 \pi t)   \end{pmatrix} \in GL^+(2, \bb R).
\]
\end{itemize}
We note that $g$ is a generator of $\pi _1(GL^+(2, \bb R))$ since $\pi _1(GL^+(2, \bb R)) \cong \pi _1(SO(2)) \cong \bb Z$. 
\end{definition}

\begin{proposition}\label{2.9}
The fundamental group $\pi _1(\mca P^+_0(X))$ is isomorphic to 
\[
\Big(\bigast _{\delta \in \Delta^+(X) } \bb Z \cdot \ell_{\delta } \Big)\times \bb Z \cdot g
\]
where $\bigast_{\delta \in \Delta^{+} } \bb Z \cdot \ell_{\delta }$ is a free product of infinite cyclic groups $\bb Z$ generated by $\ell_{\delta}$. 
\end{proposition}

\begin{proof}
Since $\mca P^+_0(X)$ is isomorphic to $\mf{D}^+_0 (X) \times GL^+(2, \bb R)$ we see $\pi _1(\mca P^+_0(X) ) \cong  \pi _1 (\mf D^+_0(X)) \times \bb Z\cdot g$. 
As we remarked before we have $\Delta (X) = \Delta ^+(X) \sqcup (-\Delta ^+(X))$. 
Hence we see
\[
\mf{D}^+_0(X) = \mf D^+(X) \setminus \bigcup _{\delta \in \Delta (X)} \< \delta \>^{\perp} = \mf D^+(X) \setminus \bigcup_{\delta \in \Delta ^+(X)} \< \delta  \>^{\perp} 
\]

Since $\mf{D}^+_0(X)$ is isomorphic to $\mf{H}_0(X)$ it is enough to show that 
\[
\pi _1(\mf{H}_0(X)) = \bigast _{\delta \in \Delta ^+} \bb Z \cdot \ell_{\delta}
\]

We choose a base point $p$ of $\mf{H}_0(X)$ so that $p = \sqrt{-1}\omega$ with $\omega ^2 \gg 2$. 
Let $\ell$ be the loop whose base point is $p$. 
Then there is a compact contractible subset $C$ whose interior $C^{\mathit{in}}$ contains $\ell$. 
Then the following set is finite:
\[
\{ p(\delta) \in C^{\mathit{in}} | \delta \in \Delta ^+(X)  \}. 
\]
Since the fundamental group of the complement of $n$-points in $C$ is the free group of rank $n$, 
we see the homotopy equivalence class of $\ell$ is uniquely given by 
\[
\ell _{\delta _1 } ^{k^1} \ell _{\delta _2}^{k_2} \cdots \ell _{\delta _m}^{k_m}
\]
where each $k_i \in \bb Z$. 
In fact if another loop $m$ is homotopy equivalent to $\ell$ by $H\colon[0,1] \times [0,1] \to \mf{H}_0(X)$, then there is a contractible compact set $C'$ such that $(C')^{\mathit{in}}$ contains the image of $H$. 
Since there are at most finite spherical point in $(C')^{\mathit{in}}$, we see the above representation is unique. 
Thus we have finished the proof. 
\end{proof}

To simplify the notations we denote $\ell _{v(A)}$ by $\ell _{A}$. 
By Remark \ref{spherical sheaf}, 
we see 
\[
\pi _1(\mf{H}_0(X)) = \< \ell _A | A\mbox{ is spherical and locally free} \> = \bigast_{A} \bb Z \ell _A. 
\]

\section{Hyperbolic structure on $\Stab ^{\mr{n}}(X)$}\label{3}

Let $\Stabd (X)$ be the connected components of $\Stab (X)$ introduced in $\S$2. 
In this section we discuss a hyperbolic structure on the normalized stability manifold $\Stab ^{\mr{n}}(X)$.

To simplify explanations of this section we always use the following notations. 
Let $(X_i, L_i)$ ($i=1,2$) be projective K3 surfaces with $\mr{NS}(X_i)= \bb ZL_i$ and let $\Phi \colon D(X_2) \to D(X_1)$ be an equivalence between them. 
The induced isometry $\mca N(X_2) \to \mca N(X_1)$ by $\Phi$ is denoted by $\Phi ^{\mca N}$. 

For a closed point $p_i \in X_i$ we set 
\[
v(\Phi (\mca O_{p_2})) = r_1 \+ n_1L_1 \+ s_1 \mbox{ and } v(\Phi^{-1}(\mca O_{p_1})) = r_2 \+ n_2L_2 \+ s_2. 
\]
Since $X_1$ and $X_2$ are Fourier-Mukai partners each other, we see $L_1^2 = L_2^2 = 2d$ for some $d \in \bb N$.

\begin{lemma}\label{3.1}
Notations being as above, 
\begin{itemize}
\item[(1)] $r_1=0$ if and only if $r_2=0$. In particular if $r_2=0$ then $\Phi ^{\mca N}(\mca O_{p_2}) = \pm v(\mca O_{p_1})  =\pm (0\+ 0\+ 1)$. 
\item[(2)] If $\Phi ^{\mca N}(\mca O_{p_2}) = 0\+ 0\+ 1$ then $\Phi ^{\mca N}$ is numerically equivalent to $(M \otimes ) ^{\mca N}$ where $M$ is in $\mr{Pic}(X_1)$ under the canonical identification $\mca N(X_2) \cong \mca N(X_1)$. 
\end{itemize}
\end{lemma}

\begin{proof}
By the symmetry it is enough to show that $r_2=0$ under the assumption $r_1=0$. 
If $r_1=0$, since $v(\Phi (\mca O_{p_2}))$ is isotropic, we see 
$n_1^2 L_1^2 = 0$. Thus $n_1=0$. 
Moreover since $v(\Phi (\mca O_{p_2}))$ is primitive, $s_1$ should be $\pm 1$. 
Hence $\Phi ^{\mca N}(0\+ 0\+ 1) = \pm (0\+ 0\+ 1)$. 
This gives the proof of the first assertion. 

Second assertion essentially follows from the argument in the proof for \cite[Corollary 10.12]{Huy}. 
Hence we recall his arguments. 

Since $\rho (X_i)=1$, there is the canonical isomorphism $f \colon \mca N(X_2) \to \mca N(X_1)$ where 
$f(0\+0\+1) =0\+0\+1, f(0\+ L_{2}\+0)=0\+L_{1}\+0$ and $f(0\+0\+1)=0\+0\+1$. 
We show that $\Phi ^{\mca N}=(\otimes M)^{\mca N }$ ($\exists M \in \mr{Pic}(X_1)$) under the canonical identification $f\colon \mca N(X_2)\to \mca N(X_1)$.

One can check easily 
\[
v(\Phi ^{\mca N} (1\+ 0\+ 0)) = 1 \+ M \+ \frac{M^2}{2} \ (\exists M \in \mr{Pic}(X_1)), 
\]
by using the facts $\< 1\+ 0 \+ 0, v(\mca O_{p_2}) \> =-1$ and $\< 1\+ 0\+ 0\>^2 =0$. 
Now consider the functor 
\[
\Psi = (\otimes M^{-1}\circ \Phi)\colon D(X_2) \to D(X_1) \to D(X_1). 
\]
Then we see $\Psi ^{\mca N}(0\+ 0\+ 1) = 0\+ 0\+ 1$ and $\Psi ^{\mca N}(1\+ 0\+ 0) = 1\+ 0\+ 0$. 
Thus $\Psi ^{\mca N}$ induces the isomorphism 
\[
\Psi ^{\mca N} \colon \mr{NS}(X_2) \to \mr{NS}(X_1). 
\]
Since $\mr{NS}(X_i) = \bb Z L_i$ we see $\Psi ^{\mca N} (L_2) = \pm L_1$. 
Since any equivalence preserves the orientations by \cite{HMS} we see 
$\Psi ^{\mca N} (L_2) = L_1$. 
This gives the proof of the second assertion. 
\end{proof}

\begin{lemma}\label{3.2}
For $(\beta _i, \omega _i) \in \mf{H}(X_i)$ $(i=1,2)$, 
we put $\beta _i  + \sqrt{-1} \omega _i = (x_i + \sqrt{-1}y_i) L_i$. 
\begin{itemize}
\item[(1)] For any $\beta _2 + \sqrt{-1}\omega _2\in \mf{H}(X_2)$, there exist $\beta _1 + \sqrt{-1}\omega _1 \in \mf{H}(X_1)$ and $\lambda \in \bb C^*$ such that $\Phi ^{\mca N} (\exp (\beta _2 + \sqrt{-1} \omega _2)) = \lambda \exp(\beta _1 + \sqrt{-1}\omega _1)$. 
\item[(2)] If $r_1\neq 0$ then $r_1 r_2 >0$. Furthermore we have
\[
x_1  + \sqrt{-1}y_1 = \frac{1}{d \sqrt{r_1 r_2}} \cdot \frac{-1}{(x_2+ \sqrt{-1}y_2)- \frac{n_2}{r_2} } +\frac{n_1}{r_1}. 
\]
In particular this gives a linear fractional transformation on $\bb H$. 
\end{itemize}
\end{lemma}

\begin{proof}
We put $\mho _2 = \exp(\beta _2 + \sqrt{-1}\omega_2)$ and $\Phi ^{\mca N} (\mho _2) = u\+ v \+ w$. 
Since we have $\mho _2^2=0$ and $\mho _2 \bar \mho _2 >0$, 
we see the following:
\begin{itemize}
\item[(a)] $v ^2 =2u w$ and 
\item[(b)] $v \bar v - u \bar w - \bar u w >0$. 
\end{itemize}
If $u=0$ then $v^2 $ should be $0$. 
Since we have $v^2 \geq 0$ by the assumption, we see $\Phi ^{\mca N}(\mho _2) = 0\+ 0\+ w$. 
This contradicts the second inequality. 
Thus $u$ should not be $0$ and we see 
\begin{eqnarray*}
\Phi ^{\mca N}(\mho _2)	&=& u (1\+ \frac{v}{u} \+ \frac{w}{u}) \\
					&=& u \Big(1\+ \frac{v}{u} \+ \frac{1}{2}\Big(\frac{v}{u}\Big)^2\Big)_.
\end{eqnarray*}
Since $\frac{v}{u}$ is in $\mr{NS}(X)\otimes \bb C$ we can put $\frac{v}{u} = (x + \sqrt{-1}y) L_1$ for some $(x, y) \in \bb R^2$. 
By the inequality of (b), we see $y \neq0$. 
Since $\Phi $ preserves the orientation by \cite{HMS}, we see $y >0$. 
Thus we have proved the first assertion.

We prove the second assertion. 
By the first assertion we put 
\[
\Phi ^{\mca N} (\exp(\beta _2 + \sqrt{-1}\omega _2)) = \lambda \exp(\beta _1 + \sqrt{-1}\omega _1). 
\]
Then we see 
\begin{eqnarray*}
\lambda	&=&	- \< \Phi ^{\mca N} (\exp(\beta _2 + \sqrt{-1} \omega _2)), v(\mca O_{p_1})  \> \\
			&=&	-\< \exp(\beta _2+ \sqrt{-1}\omega _2),  v(\Phi ^{-1}(\mca O_{p_1}) )\> \\
			&=& - Z_{(\beta _2, \omega _2)} (\Phi ^{-1}(\mca O_{p_1})), 
\end{eqnarray*}
and
\begin{eqnarray*}
-1		&=&	\< \exp(\beta _2 + \sqrt{-1}\omega _2), v (\mca O_{p_2}) \>\\
		&=&	\< \Phi ^{\mca N}(\exp(\beta _2 + \sqrt{-1}\omega _2)), v(\Phi (\mca O_{p_2})) \> \\
		&=& \lambda \cdot Z_{(\beta _1, \omega _1)} (\Phi (\mca O_{p_2})). 
\end{eqnarray*}
Thus we have 
\[
1 = Z_{(\beta _2, \omega _2)}(\Phi ^{-1}(\mca O_{p_1}))  \cdot  Z_{(\beta _1, \omega _1)} (\Phi (\mca O_{p_2})) 
\]

By Lemma \ref{3.1} we see $r_1 \neq 0$ and $r_2 \neq 0$. 
Now recall Remark \ref{Kaw4.0}. 
Since $v(\Phi (\mca O_{p_2}))^2= v(\Phi ^{-1}(\mca O_{p_1}))^2 =0$, we have 
\[
Z_{(\beta _2, \omega _2)} (\Phi ^{-1}(\mca O_{p_1})) = \frac{r_2}{2} \Big( y_2+ \sqrt{-1} \big( \frac{n_2}{r_2} - x_2 \big)   \Big)^2 L_2^2
\]
and 
\[
Z_{(\beta _1, \omega _1)} (\Phi (\mca O_{p_2})) = \frac{r_1}{2} \Big( y_1+ \sqrt{-1} \big( \frac{n_1}{r_1} - x_1 \big)   \Big)^2 L_1^2. 
\]
Since $L_1^2 = L_2 ^2 = 2d$ 
we see 
\begin{equation*}
(x_1- \frac{n_1}{r_1}) + \sqrt{-1}y_1 = \frac{\pm 1}{d \sqrt{r_1r_2}} \cdot \frac{1}{(x_2- \frac{n_2}{r_2})+\sqrt{-1}y_2}.  \label{eqn1}
\end{equation*}
Since the left hand side is in the upper half plane $\bb H$, 
$\sqrt{r_1r_2}$ should be a real number. 
Thus we see $r_1 r_2 >0$. 
Furthermore, since the imaginary part of the left hand side is positive we have
\[
(x_1- \frac{n_1}{r_1}) + \sqrt{-1}y_1 = \frac{-1}{d \sqrt{r_1r_2}} \cdot \frac{1}{(x_2- \frac{n_2}{r_2})+\sqrt{-1}y_2}.  
\]
Thus we have finished the proof. 
\end{proof}

Recall that $\Stab ^{\mr{n}}(X) = \Stabd (X) / \tGL^+(2, \bb R)$.

\begin{theorem}\label{thm1}
Assume that $\rho (X) =1$. 
\begin{itemize}
\item[(1)] $\Stab ^{\mr{n}} (X)$ is a hyperbolic 2 dimensional manifold. 
\item[(2)] Let $Y$ be a Fourier-Mukai partner of $X$ and $\Phi \colon D(Y) \to D(X)$ an equivalence. Suppose that $\Phi$ preserves the distinguished component. 
Then the induced morphism $\Phi _*^{\mr{n}} \colon \Stab ^{\mr{n}}(Y) \to \Stab ^{\mr{n}} (X)$ is an isometry with respect to the hyperbolic metric. 
\end{itemize}
\end{theorem}

\begin{proof}
By Corollary \ref{covering2}, we have the normalized covering map
\[
\pi^{\mr{n}}\colon \Stab^{\mr{n}}(X) \to \mf{H}_0(X).
\]
Since $\mf{H}_0(X)$ is isomorphic to the open subset of $\bb H$ by Lemma \ref{key lemma}, we can define the hyperbolic metric on $\mf{H}_0(X)$ which is given by 
\[
ds^2 = \frac{dx^2 + dy^2}{y^2}, 
\]
where $x+\sqrt{-1}y \in \bb H$. 
Since $\pi^{\mr{n}}$ is a covering map, we can also define the hyperbolic metric on $\Stab^{\mr{n}}(X)$. 
Thus $\Stab^{\mr{n}}(X)$ is hyperbolic. 
%

Now we prove the second assertion. 
If $v(\Phi (\mca O_y))$ is not $\pm (0\+ 0\+ 1)$ by Lemma \ref{3.2}, we see that the induced morphism between $\mf{H}_0(Y) \to \mf {H}_0(X)$ is given by the linearly fractional transformation. 
Since $\pi^{\mr{n}}$ is an isometry, $\Phi^{\mr{n}}_*$ is also an isometry. 
Suppose that $v(\Phi (\mca O_y)) = \pm(0\+ 0\+ 1)$. 
If necessary by taking a shift $[1]$ which gives the trivial action on $\mf{H}(X)$  we can assume that $v(\Phi (\mca O_y)) = 0\+ 0\+ 1$. 
Then, by Lemma \ref{3.1}, the induced action on $\bb H$ is given by a parallel transformation $z \mapsto z +n$ for some $n \in \bb Z$. 
Thus we have finished the proof. 
\end{proof}

\section{Simply connectedness of $\Stab ^{\mr{n}}(X)$}\label{4}

In this section we always assume $\rho (X) =1$. 
Then, as was shown in the previous section, $\Stab ^{\mr{n}}(X)$ is a hyperbolic manifold.  
By using the hyperbolic structure, we shall discuss the simply connectedness of $\Stabd (X)$. 
Namely we show the following:

\begin{theorem}\label{thm2}
The following conditions are equivalent. 
\begin{itemize}
\item[(1)] $\Stabd(X)$ is simply connected. 
\item[(2)] $\Stab ^{\mr{n}}(X)$ is isomorphic to the upper half plane $\bb H$. 
\item[(3)] $W(X)$ is isomorphic to the free group generated by $T_A^2$:
\[
W(X)= \bigast_{A} (\bb Z\cdot T_A^2)  ,  
\]
where $A$ runs through all spherical locally free sheaves. 
\end{itemize}
\end{theorem}

\begin{proof}
We first show that $\Stabd (X)$ is simply connected if and only if $\Stab ^{\mr{n}} (X)$ is simply connected. 
Since the right action of $\tGL^+(2, \bb R)$ on $\Stabd (X)$ is free, the natural map 
\[
\Stabd (X) \to \Stab ^{\mr{n}} (X)
\]
gives the $\tGL^+ (2, \bb R)$-bundle on $\Stab ^{\mr{n}} (X)$. 
Thus there is an exact sequence of fundamental groups:
\[
\begin{CD}
\pi _1(\tGL^+(2, \bb R)) @>>> \pi _1 (\Stabd (X)) @>>> \pi _1(\Stab ^{\mr{n}} (X)) @>>> 1. 
\end{CD}
\]
Since $\tGL^+(2, \bb R)$ is simply connected  
we see that $\pi _1 (\Stabd (X) ) =\{1\}$ if and only if $\pi _1 (\Stab ^{\mr{n}} (X))=\{1\}$.

Since $\Stab ^{\mr{n}} (X)$ is a hyperbolic and complex manifold, $\Stab ^{\mr{n}}(X)$ is isomorphic to $\bb H$ if and only if $\pi _1 (\Stab ^{\mr{n}} (X))=\{ 1\}$ by Riemann's mapping theorem. 
Thus we have proved that the first condition is equivalent to the second one. 

We secondly show the first condition is equivalent to the third one. 
Let $\mr{Cov}(\pi)$ be the covering transformation group of $\pi \colon \Stabd (X) \to \mca P^+_0(X)$. 
We put $\tilde W(X)$ by the group generated by $W(X)$ and the double shift $[2]$. 
Note that $\tilde W(X)$ is isomorphic to $W(X) \times \bb Z \cdot [2]$. 

We claim that $\tilde W(X)$ is isomorphic to $\mr{Cov}(\pi)$. 
Recall that all spherical sheaf $A$ on $X$ with $\rho (X)=1$ is $\mu$-stable by Remark \ref{spherical sheaf}. 
Hence any $\Phi \in \tilde W(X)$ gives a trivial action on $H^*(X, \bb Z)$ and preserves the connected component $\Stabd (X)$. 
Thus $\Phi $ gives the covering transformation by \cite[Theorem 13.3]{Bri}. 
Thus we have the group homomorphism $\tilde W(X) \to \mr{Cov}(X)$. 
In particular by Proposition \ref{3.3}, we see this morphism is a surjection. 
Furthermore as is shown in \cite[Theorem 13.3]{Bri}, this is injective. 
Thus we have proved our claim.

Since the covering $\pi \colon \Stabd (X) \to \mca P^+_0(X)$ is a Galois covering, we have the exact sequence of groups:
\[
\begin{CD}
1@>>> \pi _1(\Stabd (X)) @>>> \pi _1 (\mca P^+_0 (X)) @>\varphi>> \mr{Cov}(\pi) @>>>1. 
\end{CD}
\]
As will be shown in Proposition \ref{grouphom} we see $\varphi(\ell _{A}) = T_A^2$ and $\varphi (g)= [2]$. 
If $\Stabd(X)$ is simply connected then $\varphi $ is the isomorphism. 
Hence $W(X)$ is a free group generated by $T_A^2$. 
Conversely if $W(X)$ is a free group generated by $T_A^2$, then $\varphi $ is an isomorphism. 
Hence $\Stabd (X)$ is simply connected.  
\end{proof}

\begin{remark}
Since the quotient map $\Stabd (X) \to \Stab ^{\mr{n}}(X)$ is a $\tGL^+(2, \bb R)$-bundle, we see that $\Stabd(X)$ is simply connected if and only if $\Stabd(X)$ is a $\tGL^+(2, \bb R)$-bundle over $\bb H$. 
Thus we can deduce the global geometry of the stability manifold $\Stabd(X)$. 
\end{remark}

\begin{remark}\label{4.3}
We give some remarks for $W(X)$. 
Recall that any equivalence $\Phi \in \Aut (D(X))$ induces the Hodge isometry $\Phi ^{H}$ of $H^*(X, \bb Z)$ in a canonical way. 
If Bridgeland's conjecture holds, the group $W(X) \times \bb Z[2]$ is the kernel $\Ker (\kappa)$ of the natural map
\[
\kappa \colon\Aut (D(X)) \to O^+_{\text{Hodge}} (H^*(X, \bb Z))\colon\Phi \to \Phi ^{H}. 
\]
Moreover $\Ker (\kappa)$ is given by $\pi _1 (\mca P^+_0(X))$. 
The freeness of $W(X)$ means any two orthogonal complements $\< \delta _1 \>^{\perp}$ and $\< \delta _2 \>^{\perp}$ (where $\delta _1$ and $\delta _2  \in \Delta (X)$) do not meet each other in $\mca P^+_0(X)$. 

In more general situations (namely the case of $\rho (X) \geq 2$) there should be some orthogonal complements such that $\< \delta _1 \>^{\perp}$ and $ \<\delta _2\>^{\perp}$ meet each other. 
Hence we expect that 
the quotient group ${\Ker (\kappa) /\bb Z \cdot [2]}$ is not a free group. 
\end{remark}

\section{Wall and the hyperbolic structure}\label{5}

Let $X$ be a projective K3 surface with Picard rank one. 
We have two goals of this section. 
The first aim is to show Proposition \ref{grouphom} which is necessary for Theorem \ref{thm2}. 
The second aim is to show that any wall is geodesic.

Now we start this section from the following key lemma. 

\begin{lemma}\label{2.3}
Any $\sigma \in \partial U(X)$ is in a general position (See also \cite[\S 12]{Bri}).  
Namely the point $\sigma$ lies on only one irreducible component of $\partial U(X)$. 
\end{lemma}

Before we start the proof, we remark that Maciocia proved a similar assertion in a slightly different situation in \cite{Mac12}. 

\begin{proof}
Suppose that there is an element $\sigma =(\mca A, Z) \in \partial U(X)$ which is not general. 
Let $W_1$ and $W_2$ be two irreducible components of $\partial U(X)$ such that $\sigma \in W_1 \cap W_2$. 
By \cite[Proposition 9.3]{Bri} we may assume $\forall \tau_1 \in  W_1\setminus \{ \sigma \}$ and $\forall \tau_2 \in W_2 \setminus \{ \sigma \}$ are in general positions in a sufficiently small neighborhood of $\sigma$. 
Hence by \cite[Theorem 12.1]{Bri} there are two $(-2)$-vectors $\delta _i \in \Delta ^+(X)$ ($i=1,2$) such that 
for any $\tau_i =(\mca A_i, Z_i) \in W_i\setminus\{ \sigma \}$ the imaginary part $\mf{Im} Z_i(\mca O_x)\overline{Z_i(\delta_i)}$ is $0$ where $i \in \{1, 2\}$ and $x\in X$. 
Since these are closed conditions, the central charge $Z$ of $\sigma$ also satisfies the following condition:
\begin{equation}
\mf{Im}Z(\mca O_x) \overline{Z(\delta_1)} = \mf{Im} Z(\mca O_x) \overline{Z(\delta _2)} =0. \label{bcondition}
\end{equation}

By the assumption $\mr{NS}(X) = \bb Z L$, 
there exists $ g \in {GL}^+(2, \bb R)$ such that 
$Z'(E) := g^{-1} \circ Z (E) = \< \exp(\beta + \sqrt{-1}\omega), v(E) \>$
where $(\beta, \omega) \in \mf{H}(X)$. 

Now we put $\delta _i = r_i \+ n_i L \+ s_i$. 
Note that $r_i \neq 0$ since $n_i^2 L_i^2 \geq 0$. 
Since $Z'(\mca O_x) =-1$ 
we see $\mf{Im}Z'(\delta _i)$ is zero by the condition (\ref{bcondition}). 
Thus we see  
\[
\frac{n_1L}{r_1} = \frac{n_2 L }{r_2} = \beta. 
\]
Since $\delta _i^2 =-2$ we see $\gcd (r_i, n_i)=1$. 
Hence we have $\delta _1 = \delta _2$. 
This contradicts $W_1 \neq W_2$. 
\end{proof}

By Lemma \ref{2.3} and \cite[Theorem 12.1]{Bri} we see $\partial U(X)$ is a disjoint union of real codimension $1$ submanifolds:
\[
\partial U(X) = \coprod_{A:\text{spherical locally free}} (W_A^+ \sqcup W_A^-), 
\]
where $W_A^+$ (respectively $W_A^-$) is the set of stability conditions whose type is $(A^+)$ (respectively $(A^-)$). 
In the following we give an explicit description of each component $W_{A}^{\pm}$.

\begin{lemma}\label{proposition}
Let $X$ be a projective K3 surface with $\mr{NS}(X) = \bb ZL$ and let $A$ be a spherical locally free sheaf. 
We put $v(A) = r_A\+ n_A L \+ s_A$ and define the set $S(v(A))$ by
\[
S(v(A)) = \{ (\beta,\omega)\in \mf{H}(X) | \beta =\frac{n_A L}{r_A}, 0< \omega ^2 < \frac{2}{r_A^2}  \}. 
\]
Then $W_A^{\pm}$ is isomorphic to $S(v(A)) \times \tGL^+(2, \bb R)$. 
In particular $W_A^{\pm}/\tGL^+(2, \bb R)$ is a hyperbolic segment spanned by two  points in $\Stab^{\mr{n}}(X)$ which is isomorphic to $S(v(A))$.  
\end{lemma}

\begin{proof}
 We have to consider two cases: $\sigma\in W_A^+$ or $\sigma \in W_A^-$. 
Since the proof is similar, we give the proof only for the case $\sigma \in W_A^+$.

Since $\sigma \in W_A^+$, the Jordan-H\"{o}lder filtration of $\mca O_x$ is given by the spherical triangle (\ref{spherical triangle})
\begin{equation}
\begin{CD}
A^{\oplus r_A}@ >>> \mca O_x @>>> T_A (\mca O_x).  
\end{CD}\label{ichi} 
\end{equation}
By taking $T_A^{-1}$ to the triangle (\ref{ichi}) we have 
\begin{equation}
\begin{CD}
A^{\oplus r_A}[1] @>>> T_A^{-1}(\mca O_x) @>>> \mca O_x.  
\end{CD}\label{ni}
\end{equation}
Thus $\mca O_x$ is $T_{A*}^{-1} \sigma$-stable. 
Hence $T_{A*}^{-1}\sigma$ is in $U(X)$. 

Now we put $T_{A*}^{-1} \sigma = \tau = (\mca A, Z)$. 
Since $Z(A[1])/ Z(\mca O_x) \in \bb R_{>0}$, we see that $\tau$ is in the set 
\[
W' = \{\sigma _{(\beta, \omega)} \in V(X)  |  \beta = \frac{n_A L}{r_A}, \frac{2}{r_A^2} < \omega^2 \} \cdot \tGL^+(2, \bb R). 
\]
Thus we see $W_A^+ \subset T_{A*}W'$. 
To show the inverse inclusion, let $\tau ' = (\mca A', Z') $ be in $W'$. 
As we remarked in Remark \ref{spherical sheaf}, $A$ is $\mu$-stable locally free sheaf. 
Then $A[1]$ has no nontrivial subobject in $\mca A'$ by \cite[Theorem 0.2]{Huy08}. 
Hence $A[1]$ is $\tau'$-stable, in particular, with phase $1$. 
Since $T_A^{-1} (\mca O_x)$ is given by the extension (\ref{ni}) of $\mca O_x$ and $A^{\+ r_A}[1]$, 
the object $T_A^{-1}(\mca O_x)$ is strictly $\tau'$-semistable. 
Thus by taking $T_A$ to the triangle (\ref{ni}), we obtain the Jordan-H\"{o}lder filtration (\ref{ichi}). 
Hence we see $W_A^+ = T_{A*}W'$.

Since the induced morphism between $\mf{H}(X)$ by $T_A$ is given by Lemma \ref{3.2}, we see 
\[
W_A^+ = T_{A*} W' \cong S(v(A)) \times \tGL^+(2, \bb R). 
\]
\end{proof}

For a spherical locally free sheaf $A$ we define the point $q = p(T_A(\mca O_x)) \in \bar {\mf{H}}(X)$ by $(\beta, \omega)= (\frac{c_1(A)}{r_A}, 0)$.  
By the simple calculation we see that 
\[
\< \exp (q) , v(T_A(\mca O_x)) \>=0. 
\]
Thus in the sense of Definition \ref{ass.pt}, $p(T_A(\mca O_x))$ could be regarded as the associated point of the isotropic vector $v(T_A(\mca O_x))$.  
In view of this we define the following notion:

\begin{definition}\label{semirigid.pt}
An associated point $p \in \bar{\mf{H}}(X)$ with a primitive isotropic vector $v \in \mca N(X)$ is the point which satisfies
\[
\< \exp(p), v  \>=0. 
\]
Clearly if $v = r\+ nL \+s$ then $p$ is given by $\frac{n}{r}$. 
In particuclar if $v = 0\+ 0\+ 1$ the associated point is $\infty \in \bar {\mf{H}}(X)$. 
We denote the point by $p(v)$. 
\end{definition}

As an application of Lemma \ref{proposition} we give the proof of a remained proposition:

\begin{proposition}\label{grouphom}
Let $\varphi \colon \pi _1(\mca P^+_0(X)) \to \mr{Cov}(\pi)$ be the morphism in the proof of Theorem \ref{thm2}. 
Then $\varphi(\ell _{A}) = T_A^2$ and $\varphi (g)=[2]$. 
\end{proposition}

\begin{proof}

We set a base point of $\pi _1 (\mf{H}_0(X))$ as $\sqrt{-1}\omega _0$ with $\omega _0^2 \gg 2$. 
We also define a base point of $\pi _1(\mca P^+_0(X))$ by $\exp(\sqrt{-1}\omega _0)$. 
Let $\sigma _0 = \sigma _{(0, \omega _0)} \in V(X)$ be a base point of the covering map $\pi \colon \Stabd (X) \to \mca P^+_0(X)$. 

Let $\ell _A\colon  [0,1] \to \mf{H}_0(X)$ be the loop defined in Definition \ref{generator} which turns round the point $p(v(A))$ 
and let $\tilde \ell _A$ be the lift of $\ell _A$ to $\Stabd (X)$. 

The second assertion is almost obvious. 
In Definitions \ref{generator} we choosed $g$ as 
\[
g\colon [0,1] \to GL^+(2, \bb R)\colon  t\mapsto \begin{pmatrix} \cos (2\pi t) & -\sin (2\pi t) \\ \sin(2\pi t)& \cos (2 \pi t) \end{pmatrix}. 
\]
Then the induced action of $g$ on $\Stabd (X)$ is given by the double shift $[2]$. 
Hence it is enough to show that $\tilde \ell _A (1) = T_{A*}^2 \sigma _0$.

Since there are no spherical point $p(\delta)$ inside the loop $\ell _A$ except for $p(v(A))$ itself, 
the intersection $\ell_A ([0,1]) \cap \pi (\partial U(X))$ consists of only one point. 
We may assume the point is given by $\ell _A(1/2)$. 
Since we have $\tilde \ell _A([0, 1/2)) \subset  U(X)$ 
we see that $\tilde \ell _A (1/2) = \tau$ is in $\partial U(X)$ and that $\tau$ is of type $(A^+)$ or $(A^-)$ by Lemma \ref{2.3} and \cite[Theorem 12.1]{Bri}. 

We finally claim that $\tau$ is of type $(A^+)$. 
To prove the claim we put 
\[
\tilde \ell _A \Big(\frac{1}{2}-\epsilon \Big) =\sigma _{\epsilon}= (\mca A_{\epsilon}, Z_{\epsilon}) \in \Stabd (X) , 
\]
for $0< \epsilon \ll 1$. 
In fact suppose to the contrary that $\tau$ is of type $(A^-)$. 
By Proposition \cite[Proposition 9.4]{Bri} we may assume both $A$ and $T_A^{-1}(\mca O_x)$ are $\sigma_{\epsilon}$-stable for any $\epsilon$. 
Since we see $\mf{Im}Z_{\epsilon}(\mca O_x)/ Z_{\epsilon}(A[2]) >0$, 
the distinguished triangle 
\[
\begin{CD}
T_A^{-1} (\mca O_x) @>>> \mca O_x @>>> A^{\+ r_A}[2]
\end{CD}
\]
gives the Harder-Narasimhan filtration of $\mca O_x$ in $\sigma_{\epsilon}$. 
This contradicts the fact that $\mca O_x$ is $\sigma_{\epsilon}$-stable. 
Hence $\ell_A (1/2)$ is of type $(A^+)$ and $\tilde \ell _A(1/2 + \epsilon) $ is in $T_{A*}^2  U(X)$. 
For $t> 1/2$, since $\ell _A$ does not meet $\pi (\partial U(X))$, we see $\tilde \ell _A(1) = T_{A*}^2 \sigma_0$. 
\end{proof}

Finally we observe so called walls in terms of the hyperbolic structure. 
As we showed in Lemma \ref{proposition} each boundary components of $\partial V(X)$ is geodesic in $\Stab^{\mr{n}}(X)$. 
More generally we show that any wall is geodesic in $\Stab^{\mr{n}}(X)$.

Let $\mca S$ be the set objects which have bounded mass in $\Stabd(X)$, 
and $B$ a compact subset of $\Stabd(X)$. 
Then by \cite[Proposition 9.3]{Bri} we have a finite set $\{ W_{\gamma}  \}_{\gamma \in \Gamma}$ of real codimension $1$ submanifolds satisfying the property in the proposition. 
For the set $\{ W_{\gamma} \}_{\gamma \in \Gamma}$ we put 
\[
\mf{W}(\mca S,B) = \Big( \bigcup_{\gamma \in \Gamma} W_{\gamma} \Big)/\tGL^+(2, \bb R). 
\]
Note that $\mf{W}(\mca S,B)$ is a subset of $\Stab^{\mr{n}}(X)$.

\begin{theorem}\label{thm3}
The set $\mf{W}(\mca S, B)$ is geodesic in $\Stab^{\mr{n}}(X)$. 
\end{theorem}

\begin{proof}
Following \cite[Proposition 9.3]{Bri} let $\mca T$ be the set of objects 
\[
\mca T =\{ A \in D(X) | \exists E \in \mca S, \exists \sigma \in B\text{ such that }m_{\sigma}(A) \leq m_{\sigma }(E)  \}. 
\]
We put the set of Mukai vectors in $\mca T$ by $I =\{ v(A) | A \in \mca T  \}$ 
and let $\gamma $ be the pair $\gamma = (v_i, v_j) \in I \times I$ which are not proportional. 
As was shown in \cite[Proposition 9.3]{Bri}, each wall component $W_{\gamma}$ is given by 
\[
W_{\gamma} = \{ \sigma =(\mca A, Z) \in \Stabd(X) | Z(v_i)/Z(v_j) \in \bb R_{>0} \}.   
\]
We put $W_{\gamma}/\tGL^+(2, \bb R)$ by $\mf{W}_{\gamma}$. 
It is enough to prove that $\mf{W}_{\gamma}$ is geodesic in $\Stab^{\mr{n}}(X)$.

Since $I$ is finite set (Recall that $\mca T$ has bounded mass) 
we can take a sufficiently large $m \in \bb Z $ so that the rank of all vectors in $T_{mL}^H(I)$ are not $0$. 
For the set $T_{mL}^H(I)$ we define $\mf W_{\gamma}^T$ by 
\[
\mf W_{\gamma }^T= \{ [\sigma] =[(\mca A, Z)] \in \Stab^{\mr{n}}(X) | Z(T_{mL}^H(v_i))/ Z(T^H_{mL}(v_j)) \in \bb R_{>0}  \}. 
\]
We may assume the central charge of $[\sigma] \in \mf W_{\gamma}^T$ is given by 
\[
Z(E)=  \< \exp(\beta+ \sqrt{-1}\omega), v(E) \>
\]
where $(\beta, \omega ) \in \mf{H}(X)$. 

We note that $\sigma\in \mf W_{\gamma}^{T}$ satisfies the following equation
\begin{equation}
\mf{Im} Z(T_{mL}^H(v_i)) \overline{Z(T_{mL}^H(v_j))}=0.  \label{circle}
\end{equation}
Then one can easily check that the equation (\ref{circle}) defines hyperbolic line in $\mf{H}(X)$. 
Since the hyperbolic structure is induced from $\mf{H}(X)$ the set $\mf W_{\gamma}^T$ is geodesic also in $\Stab^{\mr{n}}(X)$.  
Since we have $T_{mL}^{\mr{n}} \mf{W}_{\gamma}^T = \mf{W}_{\gamma}$ the set $\mf W_{\gamma}$ is also geodesic in $\Stab^{\mr{n}}(X)$ by Theorem \ref{thm1}. 
\end{proof}

\section{Revisit of Orlov's theorem via hyperbolic structure}\label{6}

In this section we demonstrate applications of the hyperbolic structure on $\Stab^{\mr{n}}(X)$. 
Mainly we prove Orlov's theorem without the global Torelli theorem but with assuming the connectedness of $\Stab(X)$ in Proposition \ref{Orlov}. 
Hence our application suggests that Bridgeland's theory substitutes for the global Torelli theorem. 

\subsection{Strategy for Proposition \ref{Orlov}}\label{strategy}
Since the proof of Proposition \ref{Orlov} is technical, we explain the strategy and the roles of some lemmas which we prepare in \S \ref{ss6.2}. 
Proposition \ref{Orlov} will be proved in \S \ref{ss6.3}.

If we have an equivalence $\Phi \colon D(Y) \to D(X)$ preserving the distinguished component then there exists $\Psi \in W(X)$ such that $(\Psi \circ \Phi)_* U(Y) \cap V(X) \neq \emptyset$ by Proposition \ref{3.3}. 
We want to take the large volume limit in the domain $(\Psi \circ \Phi)_* U(Y) \cap V(X)$. 
Because of the complicatedness of the set $V(X)$, we consider the subset $V(X)_{>2}=\{ \sigma _{(\beta, \omega)} \in V(X) | \omega ^2 >2 \}$ and focus on the domain $D_{>2}= (\Psi \circ \Phi)_* U(Y) \cap V(X)_{>2}$. 

To take the large volume limit, we have to know the shape of the domain $D_{>2}$. 
To know the shape of $D_{>2}$ we have to see where the boundary $(\Psi \circ \Phi )_* \partial U(Y)$ appears in $\Stabd(X)$. 
As we showed in Lemma \ref{proposition}, any connected component of $\partial U(Y)$ is the product of $\tGL^+(2, \bb R)$ and  a hyperbolic segment spanned by two associated points. 
Since any equivalence $D(Y) \to D(X)$ induces an isometry between the normalized spaces $\Stab^{\mr{n}}(Y)  \to \Stab^{\mr{n}}(X)$ by Theorem \ref{thm1}, 
we see that the image $(\Psi \circ \Phi )_* \partial U(Y)$ is also the products of $\tGL^{+}(2, \bb R)$ and hyperbolic segments spanned by two associated points (See also Lemma \ref{6.1} below). 
This is the reason why the hyperbolic metric on $\Stab^{\mr{n}}(X)$ is important for us.

Here we have to recall that $\Stabd(X)$ is conjecturally $\tGL^+(2, \bb R)$-bundle over the upper half plane $\bb H$. 
Since we don't have the explicit isomorphism $\Stab^{\mr{n}}(X) \to \bb H$ yet, 
it is impossible to observe the place $(\Psi \circ \Phi )_* \partial U(Y)$ in $\Stabd(X)$. 
Instead of this observation, we study the numerical information of $(\Psi \circ \Phi )_* \partial U(Y)$, namely the image of $(\Psi \circ \Phi )_* \partial U(Y)$ by the quotient map $\pi _{\mf H} \colon \Stabd(X) \to  \mca P^+_0(X) \to\mf{H}_0(X)$. 

Set $\mf{W}= \pi _{\mf H} \big((\Psi \circ \Phi )_* \partial U(Y) \big)$. 
As we showed in Lemma \ref{proposition}, $\mf W$ is the disjoint sum of hyperbolic segments. 
As we show in Lemma \ref{6.2} later, there are two types (I) and (II) of components of $\mf{W}$. 
The type (I) is a hyperbolic segment which does not intersect the domain $\pi _{\mf H}(V(X)_{>2})$ 
and the type (II) is a hyperbolic segment which does intersect $\pi _{\mf H}(V(X)_{>2})$. 
Recall that our basic strategy is to take the limit in the domain $V(X)_{>2}$. 
If the family of type (II) components is unbounded in $\pi_{\mf H}(V(X)_{>2})$, it may be impossible to take the large volume limit.  
Hence we have to show the boundedness of type (II) components (Proposition \ref{6.3} and Corollary \ref{6.4}). 

\subsection{Technical lemmas}\label{ss6.2}

We prepare some technical lemmas. 
Throughout this section we use the following notations. 

For a K3 surface $(X, L)$ we put $L^2 =2d$. 
Suppose that $E \in D(X)$ satisfies $v(E)^2=0$ and $A \in D(X)$ is spherical. 
We put their Mukai vectors respectively 
\[
v(E) = r_E \+ n_E L \+ s_E\text{ and }v(A)= r_A \+ n_A L \+ s_A. 
\]
We denote $(\beta, \omega ) \in \bar{\mf{H}}(X)$ by $(xL, yL)$.

The main object is the following set 
\[
\mf{W}(A,E) = \{ (\beta, \omega)\in \bar{\mf{H}}(X) | \mf{Im}Z_{(\beta, \omega)}(E) \overline{Z_{(\beta, \omega)}(A)} =0  \}. 
\]
One can easily check that the condition $\mf{Im}Z_{(\beta, \omega)}(E) \overline{Z_{(\beta, \omega)}(A)} =0 $ is equivalent to 
\[
N_{A,E}(x,y)= \lambda _E (\frac{-1}{r_A} + dr_A y^2 - \frac{d \lambda _A^2}{r_A})- \lambda _A (d r_E y^2 - \frac{\lambda _E^2}{r_E} )=0, 
\]
where $\lambda _E = n_E - r_E x$ and $\lambda _A = n_A - r_A x$. 
We also have 
\begin{equation}
N_{A, E}(x,y)= d (r_A n_E -r_E n_A)y^2 + d\lambda _E \lambda _A (\frac{n_E}{r_E}- \frac{n_A}{r_A}) -\frac{\lambda _E}{r_A}.  \label{NAE}
\end{equation}

\begin{lemma}\label{6.1}
Suppose that $0< r_E $ and $\frac{n_E}{r_E}\neq \frac{n_A}{r_A} $. 
Then $\mf{W}(A,E)$ is the half circle passing through the following 4 points:
\[
(x, y)=(\alpha _E, 0), (\frac{n_E}{r_E}, 0), (\frac{n_A}{r_A}, \frac{1}{\sqrt{d}|r_A|})\text{ and }(\alpha _A, \frac{1}{\sqrt{d}|r_A|}),
\]
where $\alpha _E = \frac{n_A}{r_A}- \frac{1}{dr_A^2(\frac{n_E}{r_E}-\frac{n_A}{r_A})}$ and $\alpha _A = \frac{n_E}{r_E}- \frac{1}{dr_A^2(\frac{n_E}{r_E}-\frac{n_A}{r_A})}$. 
In particular the set $\mf{W}(A, E)$ is a hyperbolic line passing through above 4 points. 
\end{lemma}

\begin{proof}
We can prove Lemma \ref{6.1} by the simple calculation of (\ref{NAE}). 
\end{proof}
In particular the first two points are associated points with respectively $T_A(E)$ and $E$. 
Hence we put them respectively 
\begin{itemize}
\item $p(T_A(E))= (\alpha _E, 0)$, 
\item $p(E)= (\frac{n_E}{r_E}, 0)$,
\item $ p(A) = (\frac{n_A}{r_A}, \frac{1}{\sqrt{d}|r_A|})$ and
\item $q=(\alpha _A, \frac{1}{\sqrt{d}|r_A|})$. 
\end{itemize}
We remark that if $\frac{n_E}{r_E} = \frac{n_A}{r_A}$ then $\mf {W}(A,E)$ is a hyperbolic line defined by $x= \frac{n_E}{r_E}$.

\begin{lemma}\label{6.2}
Suppose that $0< r_E $ and $0< \frac{n_E}{r_E}- \frac{n_A}{r_A} $. 
Then there two types of the configuration of the above four points on $\mf{W}(A,E)$:
\begin{itemize}
\item[(I)] If $\frac{1}{d|r_A|} \leq \frac{n_E}{r_E}- \frac{n_A}{r_A}$ then we have 
$\alpha _E < \frac{n_A}{r_A}\leq  \alpha _A < \frac{n_E}{r_E}$. See also Figure \ref{typeI} below. 
\item[(II)] If $0< \frac{n_E}{r_E}- \frac{n_A}{r_A} < \frac{1}{d|r_A|} $ then we have $\alpha _E < \alpha _A < \frac{n_A}{r_A} < \frac{n_E}{r_E}$. See also Figure \ref{typeII} below. 
\end{itemize}
\end{lemma}

\begin{figure}[htbp]
 \begin{minipage}{0.6\hsize}
  \begin{center}
   \includegraphics[width=60mm, clip]{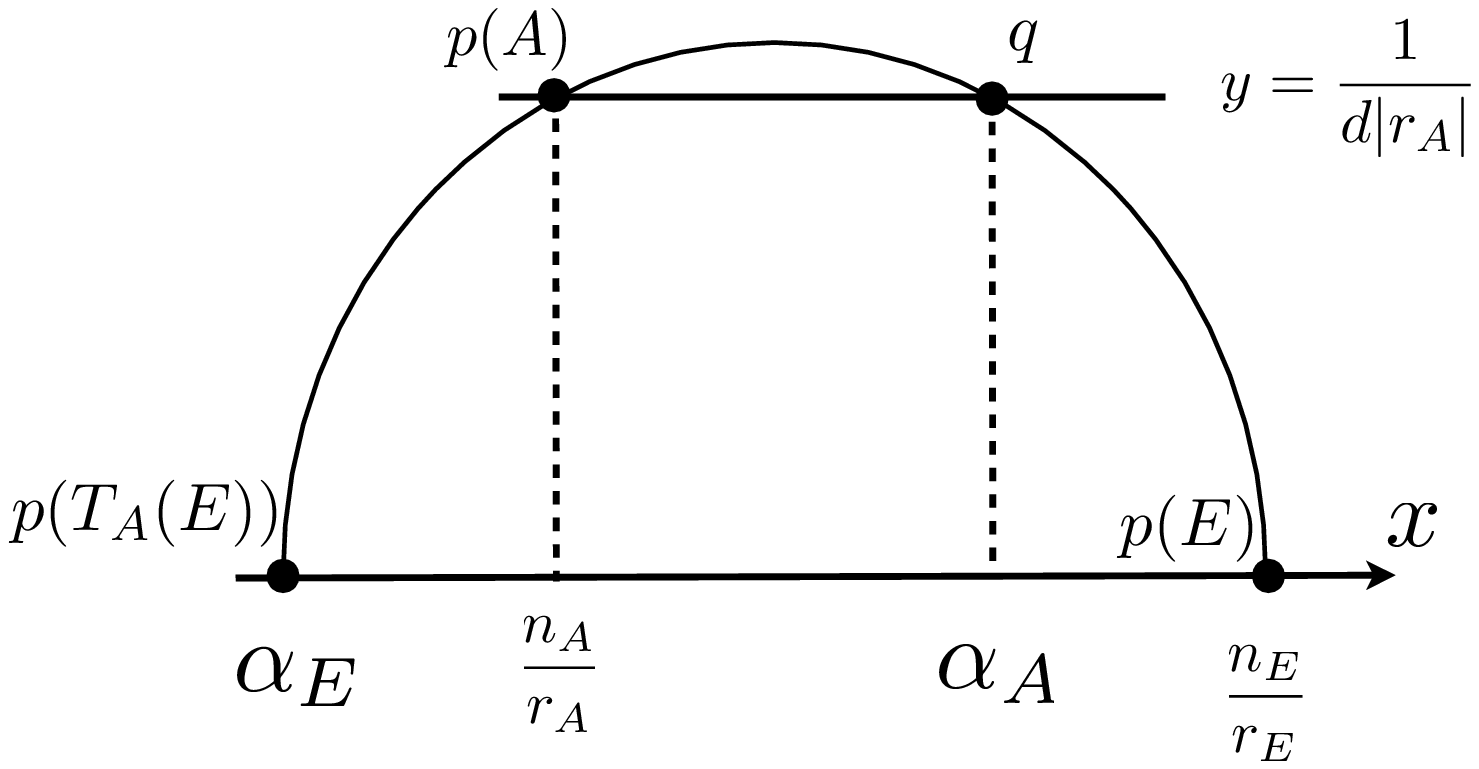}

  \caption{figure for type (I) in Lemma \ref{6.2}}
  \label{typeI}
   \includegraphics[width=60mm,clip]{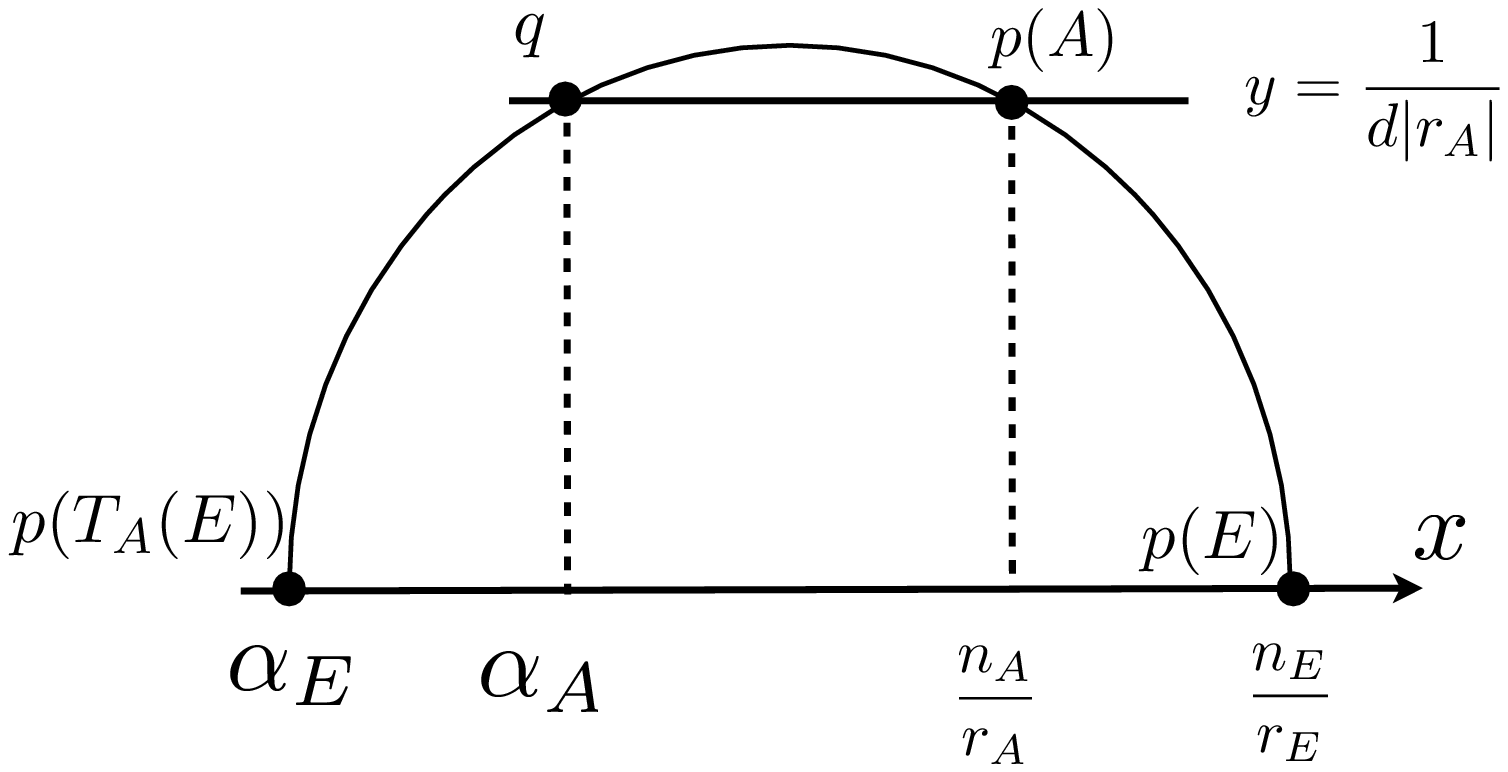}
  \end{center}
  \caption{figure for for type (II) in Lemma \ref{6.2}}
  \label{typeII}
 \end{minipage}
\end{figure}

\begin{proof}
Similarly to Lemma \ref{6.1} we could prove the assertion by simple calculations. 
\end{proof}

Let $\Phi \colon D(Y) \to D(X)$ be an equivalence preserving the distinguished component. 
Suppose $E= \Phi (\mca O_y)$. 
By Lemma \ref{proposition}, $\pi _{\mf H} (\Phi _* \partial U(Y))$ is the direct sum of hyperbolic segments $\overline{p(A) p(T_A(E))}$ spanned by two points $p(A)$ and $p(T_A(E))$. 
Clearly the segment $\overline{p(A)  p(T_A(E))}$ is a subset of $\mf{W}(A,E)$. 
Following Lemma \ref{6.2} we have the disjoint sum :
\begin{equation}
\pi _{\mf H} (\Phi _* \partial U(Y)) = \coprod_{\mr{type(I)}}\overline{p(A') p(T_{A'}(E)) } \sqcup  \coprod_{\mr{type(II)}}\overline{p(A)  p(T_{A}(E))}   . \label{daiji}
\end{equation}
Since the type (II) segments become obstructions when we take the large volume limit in $V(X)_{>2}$. 
Hence we have to show the boundedness of type (II) segments. 
To show this, we give an upper bound of the diameter of the type (II) half circle $\mf{W}(A,E)$ in the following proposition. 
Clearly from Lemma \ref{6.1} the diameter is given by $\frac{n_E}{r_E}- \alpha _E$.

\begin{proposition}\label{6.3}
Suppose that $r_E>0$ and $0< \frac{n_E}{r_E}- \frac{n_A}{r_A} < \frac{1}{\sqrt{d}|r_A|}$. 
Then we have 
\[
0 < \frac{n_E}{r_E} - \alpha _E \leq \frac{1}{r_E} + \frac{r_E}{d}. 
\]
\end{proposition}

\begin{proof}
By the assumption one easily sees $r_A \cdot( r_A n_E - r_E n_A) >0$. 
Hence we see 
\begin{eqnarray}
\frac{n_E}{r_E} - \alpha _E	&=& \Big( \frac{n_E}{r_E} - \frac{n_A}{r_A} \Big) + \frac{1}{dr_A^2 \Big( \frac{n_E}{r_E} - \frac{n_A}{r_A} \Big)} \notag \\
							&=& \Big| \frac{1}{r_A} \Big| \cdot \Big( \frac{ |r_An_E - r_E n_A |}{r_E} + \frac{r_E}{d | r_A n_E - r_E n_A|}  \Big) \notag \\
							&\leq&	\frac{|r_A n_E - r_E n_A|}{r_E} + \frac{r_E}{d |r_A n_E- r_E n_A|}. \label{aa} 
\end{eqnarray}
By the assumption we have 
\[
\frac{|r_A n_E - r_E n_A|}{r_E} < \frac{r_E}{d |r_A n_E - r_E n_A|}. 
\]
Since the continuous function $f(t)= \frac{1}{t}+\frac{t}{d}$  on $\bb R_ {>0}$ is an increasing function for $\frac{1}{t} < \frac{t}{d}$. 
Since we have $\frac{r_E}{| r_A n_E - r_E n_A|} \leq r_E$ the following inequality holds: 
\[
(\ref{aa}) \leq \frac{1}{r_E} + \frac{r_E}{d}. 
\]
Thus we have proved the inequality. 
\end{proof}

The following corollary is a simple paraphrase of Proposition \ref{6.3}. 
However it is crucial for the proof of our main result, Proposition \ref{Orlov}. 

\begin{corollary}\label{6.4}
Let $\Phi \colon D(Y) \to D(X)$ be an equivalence which preserves the distinguished component. 
Set $v(\Phi(\mca O_y)) = r \+ n L_X \+ s$ and $L_X ^2 =2d$ and assume $r >0$. 
Then the image $\pi _{\mf H} (\Phi _* \partial U(Y))$ is in the following shaded closed region $R(Y, \Phi)$ where $\pi_{\mf H}\colon \Stab^{\mr{n}} (X) \to \mf{H}_0(X)$  $:$ 
\begin{align*}
R(Y, \Phi) =\{ (xL_X, yL_X) \in \mf{H}(X)|  \Big(x-\frac{n}{r} + \frac{1}{2}\Big( \frac{1}{d}+ \frac{r}{d} \Big) \Big)^2 + y^2 \leq  \frac{1}{4}\Big( \frac{1}{d}+ \frac{r}{d} \Big)^2 ,  \\    
\Big(x-\frac{n}{r} - \frac{1}{2}\Big( \frac{1}{d}+ \frac{r}{d} \Big) \Big)^2 + y^2 \leq  \frac{1}{4}\Big( \frac{1}{d}+ \frac{r}{d} \Big)^2
 \mbox{ or }y^2 \leq \frac{1}{d}   \}.
\end{align*}
\begin{figure}[htbp]
  \begin{center}
   \includegraphics[width=70mm, clip]{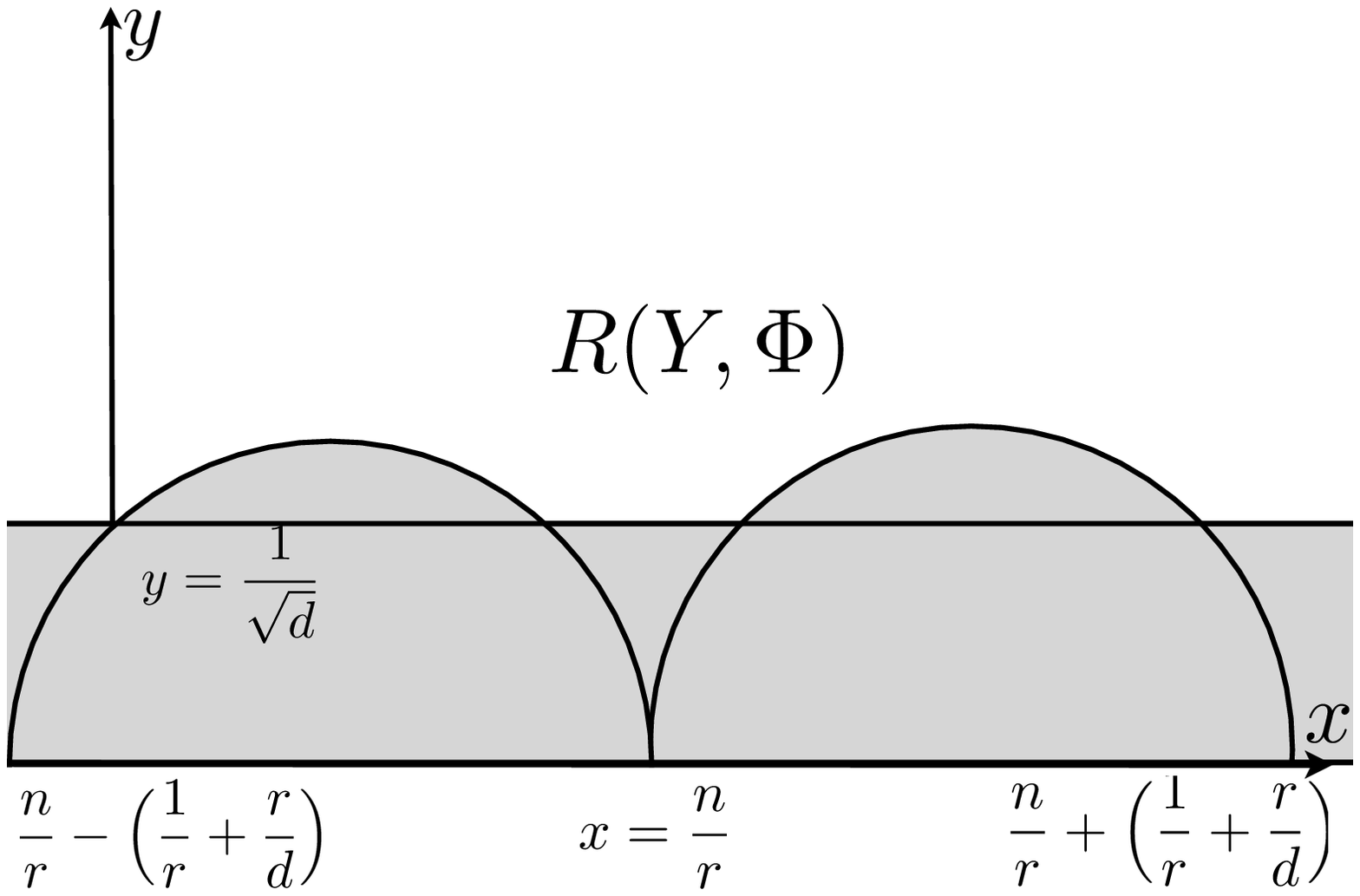}
  \caption{Figure for the region $R(Y, \Phi)$.}
  \label{RY}
\end{center}
\end{figure}
\end{corollary}

\begin{proof}
As we explained in (\ref{daiji})
, we see that 
\[
\pi _{\mf H} (\Phi _* \partial U(Y)) = \coprod _{\mr{type(I)}}\overline{p(A') p(T_{A'}(\Phi(\mca O_y)))} \sqcup \coprod _{\mr{type(II)}}\overline{p(A) p(T_{A}(\Phi(\mca O_y)))}
\]
where $A$ and $A'$ are spherical object of $D(X)$. 
Clearly type (I) hyperbolic segments $\overline{p(A') p(T_{A'}(\Phi(\mca O_y)))}$ are in the following region:
\[
\{ (xL_X, y L_X ) |   y^2 \leq \frac{1}{d} \}. 
\]
By Proposition \ref{6.3}, the type (II) hyperbolic segments are in the region
\begin{align*}
\{ (xL_X, yL_X) \in \mf{H}(X)|  \Big(x-\frac{n}{r} + \frac{1}{2}\Big( \frac{1}{d}+ \frac{r}{d} \Big) \Big)^2 + y^2 \leq  \frac{1}{4}\Big( \frac{1}{d}+ \frac{r}{d} \Big)^2 \mbox{ or } \\    
\Big(x-\frac{n}{r} - \frac{1}{2}\Big( \frac{1}{d}+ \frac{r}{d} \Big) \Big)^2 + y^2 \leq  \frac{1}{4}\Big( \frac{1}{d}+ \frac{r}{d} \Big)^2 \}. 
\end{align*}
This gives the proof. 
\end{proof}

\subsection{Revisit of Orlov's theorem}\label{ss6.3}

We prove the main result of this section. 
\begin{proposition}\label{Orlov}
Let $(X,L_X)$ be a projective K3 surface with $\rho (X)=1$ and $(Y, L_Y)$ a Fourier-Mukai partner of $(X,L_X)$. 
If an equivalence $\Phi \colon  D(Y) \to D(X)$ preserves the distinguished component, 
then $Y$ is isomorphic to the fine moduli space of Gieseker stable torsion free sheaves. 
\end{proposition}

\begin{proof}
We first put $L_X^2 = L_Y^2 =2d$ and $v_0 = v(\Phi (\mca O_y)) = r \+ n L_X \+ s$. 
If necessary by taking $T_{\mca O_X}$ and $[1]$, we may assume $r>0$. 
We denote the composition of two morphisms $\Stabd(X) \to \mca P^+_0(X) \to \mf{H}_0(X)$ by $\pi _{\mf H}$. 
By the assumption we have $\Phi _*U(Y) \subset \Stabd(X)$. 

We can take a stability condition $\tau \in U(Y)$ so that $\pi _{\mf H}(\Phi _* \tau)= (\beta_0, \omega _0)= (aL_X, bL_X)$ with 
\begin{itemize}
\item[(i)] $ a  <  \frac{n}{r} -\big(\frac{1}{r} + \frac{r}{d} \big) $ and 
\item[(ii)] $2 < \omega _0^2$. 
\end{itemize}
By the second condition (ii) and Lemma \ref{proposition} we see $\pi _{\mf H} \circ \Phi _*(\tau)$ does not lie on $\pi _{\mf H}( \partial U(X))$. 
Hence $\Phi_* (\tau)$ is in a chamber of $\Stabd(X)$ by Proposition \ref{3.3}. 
Hence we see
\[
\exists \Psi \in W(X) \times \bb Z [2]\text{ such that }(\Psi \circ \Phi) _* (\tau) \in U(X). 
\]

Now we put $\Phi ' = \Psi \circ \Phi$ and take $\sigma_0 \in V(X)$ as $\sigma_{(\beta _0, \omega _0)}$. 
Since $\Phi ' _*(\tau)$ and $\sigma_0$ belong to the same $\tGL^+(2, \bb R)$-orbit, $\sigma _0$ is in $V(X) \cap \Phi '_* (U(Y))$. 
We define a family $\mca F$ of stability conditions as follows:
\[
\mca F = \{ \sigma _{(\beta_0 , t \omega _0)} \in V(X) | 1 <t \in \bb R  \}. 
\]
Then we see $\pi _{\mf H}(\mca F) \cap R(Y,\Phi' )=\emptyset$ by Corollary \ref{6.4}. 
Hence $\mca F$ does not meet $\Phi '_* (\partial U(Y))$. 
Since $\sigma_0 \in \Phi '_* (U(Y))$ we see $\mca F  \subset \Phi '_* (U(Y))$ and the object $\Phi'(\mca O_y)$ is $\sigma$-stable for all $\sigma \in \mca F$. 
By Bridgeland's large volume limit argument \cite[Proposition 14.2]{Bri} we see that $\Phi ' (\mca O_y) $ is a Gieseker semistable torsion free sheaf\footnote{Since we are assuming $\rho (X)=1$, the Gieseker stability is equivalent to the twisted stability. }. 
Moreover by \cite[Proposition 3.14]{Muk} (or the argument of \cite[Lemma 4.1]{Kaw10}) $\Phi '(\mca O_y)$ is Gieseker stable. 
Since $v_0 = v(\Phi '(\mca O_y))$ is isotropic and there is $ u \in \mca N(X)$ such that $\< v_0,u \>=1$, there exists the fine moduli space $\mca M$ of Gieseker stable sheaves (See also \cite[Lemma 10.22 and Proposition 10.20]{Huy}). 
Hence $Y$ is isomorphic to $\mca M$. 
\end{proof}

\begin{remark}\label{finalrmk2}
Clearly the key ingredient of Proposition \ref{Orlov} is Corollary \ref{6.4}. 
The role of Corollary \ref{6.4} is to detect the place of the numerical image of walls $\pi _{\mf H} (\Phi (\partial U(Y)))$. 
Without Theorems \ref{thm1} and \ref{thm3}, it was difficult to detect the place of $\pi _{\mf H} (\Phi (\partial U(Y)))$. 
By virtue of these theorems, the problem is reduced to the problem with two associated points $p(A)$ and $p(T_A(\Phi (\mca O_y)))$. 
\end{remark}

%
%
%
%
%

\begin{remark}\label{finalrmk}
We explain the relation between author's work and Huybrechts's question in \cite{Huy08}. 

In \cite[Proposition 4.1]{Huy08}, it was proven that all non-trivial Fourier-Mukai partners of projective K3 surfaces are given by the fine moduli spaces of $\mu$-stable locally free sheaves (See also \cite[Proposition 4.1]{Huy08}). 
We note that this proposition holds for all projective K3 surfaces. 
If the Picard rank is one, the proof of the proposition is based on the lattice argument. 
In the proof of \cite[Proposition 4.1]{Huy08} Huybrechts asks whether there is a geometric proof. 

In the previous work \cite[Theorem 5.4]{Kaw11}, the author gave an answer of Huybrechts's question, that is a geometric proof. 
However our proof is not completely independent of lattice theories,  
because it is based on Orlov's theorem which strongly depends on the global Torelli theorem.

As a consequence of Proposition \ref{Orlov} and \cite[Theorem 5.4]{Kaw11}, 
we could give the another proof of \cite[Proposition 4.1]{Huy08} which is completely independent of the global Torelli theorem with assuming the connectedness of $\Stab(X)$. 
\end{remark}


\section{Stable complexes in large volume limits}\label{7}
Let $A$ be a spherical sheaf in $D(X)$. 
At the end of this paper we discuss the stability of the complex $T_A^{-1}(\mca O_x)$ in the large volume limit\footnote{We remark that $T_A^{-1}(O_x)$ is a 2-terms complex such that $H^0(T_A^{-1}(\mca O_x)) = \mca O_x$ and $H^{-1}(T_A^{-1}(\mca O_x))= A^{\+ r_A}$.}. 
More precisely we shall show that $T_A^{-1}(\mca O_x)$ is $\sigma _{(\beta, \omega )}$-stable if $\beta \omega < \mu_{\omega}(A)$ and $\omega^2 >2$. 
The possibility of stable complexes in the large volume limit is predicted in \cite[Lemma 4.2 (c)]{Bay09}.

For the vector $v(A) = r_A \+ n_A L \+ s_A$ we define the subset $\mf{D}_A \subset \mf{H}(X)$ as follows: 
\[
\mf{D}_A= \{ (xL , yL) \in \mf{H}(X) |  (x-\frac{n_A}{r_A})^2 + (y-\frac{1}{2\sqrt{d}r_A})^2  < \frac{1}{4dr_A^2} \}
\]


\begin{lemma}\label{D2}
Notations being as above. 
In the domain $\mf{D}_A$, there are no spherical point $p(\delta)$ with $(-2)$-vectors $\delta$. 
Moreover $\mf{D}_A$ does not intersect $\pi _{\mf H} \circ T_{A*} (\partial U(X))$. 
\end{lemma}

\begin{proof}
By the spherical twist $T_A$, we have the diagram:
\[
\begin{CD}
\Stab^{\mr{n}}(X) @>T_{A*}^{\mr{n}}>> \Stab^{\mr{n}}(X)\\
@VVV @VVV \\
\mf{H}_0(X) @>T_A^{\mf H}>> \mf{H}_0(X).  
\end{CD}
\]
By Lemma \ref{3.2}, $T_A^{\mf H}$ is given by the liner fractional transformation
\[
T_A^{\mf H}(x+\sqrt{-1}y) = \frac{1}{dr_A}\cdot \frac{-1}{x+ \sqrt{-1}y- \frac{n_A}{r_A}} + \frac{n_A}{r_A}. 
\]
We remark that $T_A^{\mf H}$ is conjugate to the transformation $z \mapsto -1/dr_A z$. 

Now we recall there are no spherical point $p(\delta)$ in the domain $\mf{H}(X)_{>2}
 = \{ (\beta, \omega ) \in \mf{H}(X) | \omega ^2 >2  \}. $
One can easily check that $T_A^{\mf{H}}(\mf{H}(X)_{>2})= \mf{D}_A$. 
Moreover it is clear that $\pi _{\mf H} (\partial U(X))  \cap \mf{H}(X)_{>2} $. 
This gives the proof. 
\end{proof}

Define the subset $D_A^+ \subset V(X) $ by 
\[
D_A^+= \{ \sigma_{(xL, yL)} \in V(X)  |  x < \frac{n_A}{r_A}, (xL , yL) \in \mf{D}_A  \}.
\]

In the following proposition, we discuss the stability of sheaves $T_A(\mca O_x)$ in the ``small'' volume limit $D_A^+$. 

\begin{proposition}\label{D1}
For any $\sigma \in D_A^+$, $T_A(\mca O_x)$ is $\sigma$-stable. 
In particular $D_A^+ \subset  T_{A*} U(X) \cap V(X)$. 
\end{proposition}

\begin{proof}
To simplify the notation we set $A(x)= T_A(\mca O_x)[-1]$. 
It is enough to show that $A(x)$ is $\sigma$-stable for all $\sigma \in D_A^+$. 

One can see $A(x)$ is the kernel of the evaluation map $\Hom(A, \mca O_x)\otimes A \to \mca O_x$ and is Gieseker stable. 
We note that there exists $\sigma \in D_A^+$ such that $A(x)$ is $\sigma$-stable by \cite[Theorem 4.4 (2)]{Kaw11}. 
In particular we see $D_A^+ \cap T_{A*}(U(X)) \neq \emptyset$. 
Hence it is enough to show $D_A^+ \cap T_{A*}(\partial U(X)) = \emptyset$. 
This is obvious by Lemma \ref{D2}. 
\end{proof}

We set 
\[
(D_A^+)^{\vee} = \{ \sigma_{(xL , yL)} \in V(X) | (yL)^2 > 2, x> \frac{n_A}{r_A}  \}. 
\]

\begin{corollary}\label{D3}
For any $\sigma \in (D_A^+)^{\vee}$, $T_A^{-1}(\mca O_x)$ is $\sigma$-stable. 
In particular $(D_A^+)^{\vee} \subset T_{A*}^{-1} (U(X))$. 
\end{corollary}

\begin{proof}
Since $D_A^+ \subset T_{A*} (U(X)) \cap U(X)$ by Proposition \ref{D1}, we see
\[
T_{A*}^{-1} (D_A^+) \subset U(X) \cap T_{A*}^{-1}(U(X)). 
\]
Since the $\sigma$-stability is equivalent to the $\sigma \cdot \tilde {g}$-stability for any $\tilde g \in \tGL^{+}(2, \bb R)$, 
it is enough to show that $T_{A*}^{-1} (D_A^+)/\tGL^+(2, \bb R) = (D_A^+)^{\vee}/\tGL^+(2,\bb R)$. 
This is obvious from Lemma \ref{3.2}. 
\end{proof}

\begin{remark}
In the article \cite[Lemma 4.2 (c)]{Bay09}, the possibility of the stable complexes in large volume limits is referred.  Hence Corollary \ref{D3} gives the proof of this prediction. 
\end{remark}

\subsection*{Acknowledgement}

The author is partially supported by Grant-in-Aid for Scientific Research (S), No 22224001.

%
%
%
%
%

\end{document}